\documentclass[11pt,reqno]{amsart}
\usepackage{amsmath}
\usepackage{amsthm}
\usepackage{graphicx}
\usepackage{amsfonts}
\usepackage{amssymb}
\usepackage{enumerate}
\usepackage{color}
\usepackage{bm}
\usepackage[margin=1.1in]{geometry}
\numberwithin{equation}{section}

\newtheorem{theorem}{Theorem}[section]
\newtheorem{lemma}[theorem]{Lemma}
\newtheorem{proposition}[theorem]{Proposition}
\newtheorem{corollary}[theorem]{Corollary}

\theoremstyle{definition}
\newtheorem{example}[theorem]{Example}
\newtheorem{definition}[theorem]{Definition}
\newtheorem*{standingassumption*}{Standing assumptions}
\newtheorem*{assumptionA*}{Assumption A}
\newtheorem*{assumptionB*}{Assumption B}
\newtheorem*{assumptionC*}{Assumption C}
\newtheorem*{assumptionD*}{Assumption D}
\newtheorem{remark}[theorem]{Remark}

\def\E{{\mathbb E}}
\def\R{{\mathbb R}}

\def\PP{{\mathbb P}}
\def\P{{\mathcal P}}
\def\RC{{\mathcal R}}

\def\M{{\mathcal M}}

\def\X{{\mathcal X}}
\def\Y{{\mathcal Y}}

\def\L{{\mathcal L}}
\def\Z{{\mathcal Z}}
\def\G{{\mathcal G}}

\def\T{{\mathbb T}}
\def\SS{{\mathcal S}}

\def\W{{\mathcal W}}

\def\A{{\mathcal A}}
\def\F{{\mathcal F}}
\def\FF{{\mathbb F}}
\def\GG{{\mathbb G}}

\def\C{{\mathcal C}}

\DeclareMathOperator*{\esssup}{ess\,sup}
\DeclareMathOperator*{\essinf}{ess\,inf}

\title{Mean field games of timing and models for bank runs}
\author{Ren\'e Carmona, Fran\c{c}ois Delarue, and Daniel Lacker}

\begin{document}

\begin{abstract}
The goal of the paper is to introduce a set of problems which we call mean field games of timing. We motivate the formulation by a dynamic model of bank run in a continuous-time setting. We briefly review the economic and game theoretic contributions at the root of our effort, and we develop a mathematical theory for continuous-time stochastic games where the strategic decisions of the players are merely choices of times at which they leave the game, and the interaction between the strategic players is of a mean field nature.
\end{abstract}

\maketitle

\section{Introduction}

Our starting point is the set of early game theoretic models for the banking system due to  Bryant \cite{Bryant} and Diamond and Dybvig \cite{diamonddybvig-bankruns} whose fundamental papers proposed banking models of a game played by depositors in which there always exist at least a good equilibrium and a bad equilibrium. Many generalizations followed, for example to include illiquidity effects and more random factors, and extended the scope of the models beyond bank runs and deposit insurance to include financial intermediation, as in the work
 \cite{RochetVives} of Rochet and Vives. There, the authors use the methodology of global games proposed by Morris and Shin in \cite{MorrisShin1998} and the differences in opinions among investors to prove existence and uniqueness of a Nash equilibrium. They go on to analyze the economic and financial underpinnings of bank runs and propose a benchmark for the role of lenders of last resort.
While still in a static framework, the work \cite{greenlin-bankruns} of Green and Lin  discusses stochastic equilibria (a.k.a. aggregate uncertainty) in a context which is very close to our notion of weak equilibrium, to be defined later in the paper.

Authors of the early game theoretic papers on bank runs quickly realized that their models exhibited what is now known as a \emph{complementarity property}. Typically, if more depositors withdraw their funds early, then the probability
of failure of the bank increases, and this further incentivizes early withdrawal. Mathematically, the eventual payoff to one depositor displays \emph{increasing differences} with respect to the actions of the others depositors. 
This property is known as complementarity, and games with this property are called 
supermodular games. 
The equilibrium theory of these games hinges on their order structure more than their analytic properties (see, for example, \cite{milgrom1990rationalizability,HuangLi}),
using machinery first developed by
Topkis \cite{topkis1979equilibrium,topkis1978minimizing} and later refined by Milgrom and Roberts \cite{milgrom1990rationalizability} and Vives \cite{vives1990nash}.

A common feature of many bank run models is the symmetric or mean-field nature of the interaction between the depositors, and the goal of our paper is to take advantage of this property to develop a general mathematical theory.
While most of the works cited above are static in nature, our interest in dynamic models of bank runs was sparked by a lecture of Olivier Gossner at a PIMS Workshop on Systemic Risk in July 2014 who attempted to extend to a continuous-time setting an earlier work of Rochet and Vives \cite{RochetVives}. In 
this model, the common source of randomness comes from the value of the investments of the bank and the possible need for fire sales to face fund redemption, while the differences in the private signals of the investors contribute to the idiosyncratic sources of noise, ruling out undesirable equilibria. 
Another continuous-time bank run model worthy of mention can be found in the  paper \cite{HeXiong} by He and Xiong where the source of randomness comes from the staggered nature of the debt maturities.

With these bank run models in mind, we propose a general class of continuous-time models we call \emph{mean field games of timing}, in which a continuum of agents strategically choose stopping times, i.e., times at which to exit the game. 
We present two different sets of results for two different regimes. Under the aforementioned complementarity property, we prove that ``mean field equilibria'' (MFE) exist and illustrate how to use them to construct approximate equilibria for the corresponding $n$-player games, and this is done for very general partial information structures. On the other hand, without complementarities, we derive an existence result for ``weak MFE'' under stronger continuity assumptions, and only in the full-information setting.
We then connect weak MFE to $n$-player games by proving two modes of convergence. On the one hand, the equilibria themselves in the $n$-player game (if they exist) converge to weak MFE as $n\rightarrow\infty$. On the other hand, a weak MFE can be used to construct approximate equilibria for the $n$-player games.

Our models are closely related to the mean field games introduced independently by Lasry and Lions \cite{lasrylionsmfg} and Caines, Huang and Malham\'e in \cite{CainesHuangMalhame}. However in our models, agents act by choosing stopping times as opposed to control processes. We adhere to a purely probabilistic approach, though in principle a PDE formulation is possible involving a variational inequality or free-boundary problem. Probabilistic methods in mean field game theory originated in \cite{carmona-delarue-mfg}, although our techniques are more closely related to the weak convergence and compactness arguments of \cite{carmonadelaruelacker-mfgcommonnoise,lacker-mfglimit,lacker-controlledmtgproblems}. While most (continuous-time) mean field game models involve agents choosing control processes as opposed to stopping times, a notable exception is the recent work of Nutz \cite{Nutz}, which studies a tractable yet versatile model for which equilibria can be computed or at least characterized quite explicitly. Section \ref{se:nutz} shows how this model fits into our framework.

Our existence result (Theorem \ref{th:existence-B}) based on monotonicity properties resembles some recent papers on games with both complementarities and a continuum of agents.
For instance, Adlakha and Johari \cite{adlakha-johari} employ some similar techniques to study a discrete time mean field game with strategic complementarities.
The work of Balbus et al. \cite{balbus2014qualitative} on static games is also quite relevant, and it even includes a discussion of discrete-time ``optimal stopping games,'' although stochastic factors are absent from their model. See also \cite{yang2013nonatomic} and its correction in \cite{balbus2015monotone} for related work on nonatomic supermodular games.
The reader interested in games with complementarities may also consult the recent work of Acemoglu and Jensen \cite{acemoglu2013aggregate,acemoglu2012robust} on \emph{aggregate games}, which closely resemble mean field games.

The technical crux of our proofs requires some new results, interesting in their own right, on progressive enlargements of filtrations \cite{jeanblanc2009immersion,bremaudyor-changesoffiltrations}, particularly related to the ``compatibility'' or ``immersion'' property (also known as the H-hypothesis) which has recently seen renewed interest in light of its many applications in credit risk models. Our work necessitates a new characterization of when a filtration enlarged progressively by a random time satisfies this compatibility property: roughly speaking, if a filtration $\FF$ is generated by a Wiener process $W$, and if it is enlarged progressively to $\GG$ in the minimal way to render a given a random time $\tau$ a stopping time, then $\FF$ is ``compatible'' with $\GG$ if and only if there exists a sequence of $\FF$-stopping times $\tau_n$ such that $(W,\tau_n)$ converges to $(W,\tau)$ in distribution.
This notion of compatibility arises naturally because of the central role played by weak convergence arguments  in our analysis; essentially the same issue appears in the papers \cite{carmonadelaruelacker-mfgcommonnoise,lacker-mfglimit}, which deal with more traditional mean field game models.

The paper is organized as follows. The next section presents the continuous-time model of bank run based on some of the ideas of \cite{RochetVives} and Gossner's lecture mentioned earlier. This is borrowed from the forthcoming book \cite{CarmonaDelarue}, and we present a streamlined version for the purpose of motivation. We use continuous time stochastic processes to model the value of the assets of the bank and the private signals of the depositors. Stylized facts from economic models of bank runs are captured in a set of assumptions about the costs and rewards to the depositors, and a mathematical problem of game of timing is articulated. Section \ref{se:results} describes a general mathematical framework generalizing the set-up of the previous section. There, we provide all the required definitions and notation, and state the first main results of the paper, under complementarity assumptions. The following Section \ref{se:weakEQ} specializes the setup further to models with continuous objective functions driven by Wiener processes. No proofs are given in these sections, only illustrations of how the abstract framework generalizes the bank run presented in section \ref{se:runs}, and how the results answer the questions raised therein.
The remainder of the paper, from Section \ref{se:proofs-supermodular} on, is devoted to the proofs of the results announced in  Sections \ref{se:results} and \ref{se:weakEQ}.
Section \ref{se:compatibility}, in particular, develops the requisite material on filtration enlargements and randomized stopping times, some of which may be of independent interest. Two appendices provide proofs of technical results which we could not find in the printed literature.

\vskip 6pt\noindent
\emph{Acknowledgements:} We would like to thank Geoffrey Zhu for enlightening discussions at an early stage of our investigation of mean field games of timing.

\section{A Model for Bank Runs}
\label{se:runs}

The nature of the balance sheet of a bank and the impact of the fire sales triggered by
depositors’ runs and the possible failure of the bank are two important elements of the analysis of bank
runs and their consequences, especially from a regulatory perspective. However, for the purpose of
our mathematical analysis, we shall simplify their roles in order to focus on the optimal
timing decisions of the investors.

Suppose the market value of the assets of a bank evolve over time according to some (real-valued) stochastic process $B=(B_t)_{t \ge 0}$, where the initial value $B_0>0$ of the bank assets is known to everyone, and in particular to the depositors.
We assume that the assets generate a flow of dividends at rate $\overline{r}$ strictly greater than the risk free rate $r$. 
These dividends are not reinvested, so they are not included in $B_t$.
The depositors are promised the same interest rate $\overline{r}$
on their deposits. The bank remains in business as long as $B_t>0$.

Let $n$ be the number of depositors. We shall eventually let $n\to\infty$ to derive a mean field game model.
For this reason, we normalize the initial deposit of each investor to $D^i_0=1/n$,
so the aggregate initial deposit is $1$. We introduce 
a (deterministic) function $L$ (typically satisfying at least $L(0)=0$ and $0 < L' < 1$), and
we think of the value $L(B_t)$ as the \emph{liquidation value} of the assets of the bank at time $t$. 
As $L$ is deterministic, it is known to everyone.

Whenever an investor tries to withdraw his or her deposit, the bank taps a credit line at interest rate $\overline{r}>r$ to pay the running investor. At time $t$, the credit line limit is equal to the liquidation value $L(B_t)$ of the bank's assets. The model is set up this way to allow the bank to pay running investors without having to tinker with its investments.

The bank is said to be safe if all depositors can be paid in full, even in case of a run.
The bank is said to have liquidity problems if the current market value of its
assets is sufficient to pay depositors, but the liquidation value is not. Finally, it is said to be insolvent 
if the current market value of its assets is less than its obligation
to depositors. 
We shall confirm below that, in the case of complete information about the 
value of the assets of the bank, depositors start to run as soon as the bank has liquidity problems, possibly long
before the bank is insolvent.

Let $T$ be a finite time horizon, for the sake of concreteness, but notice that the story to follow makes just as much sense with $T=\infty$ or even when $T$ is an appropriately random time.
At time $T$, the bank's assets mature and generate a single payoff $B_T$ which can be used to pay the credit line and the depositors. 
Cash flows stop at time $T$. At that time,
\begin{itemize}\itemsep=-2pt
\item if $B_T\ge 1,\;\;\;$ the bank is safe and everybody is paid in full;
\item if $B_T < 1,\;\;\;$ the bank cannot pay everybody in full, there is an \emph{exogenous default}.
\end{itemize}
This is not the only way the bank can default. Indeed there is the possibility of an \emph{endogenous default} at time $t<T$ if the aggregate amount of withdrawals by running depositors exceeds $L(B_t)$. 
Let us denote by $\tau^i$ the time at which depositor $i$ tries to withdraw his or her deposit,
and by $\overline{\mu}^n$ the empirical distribution of these times, i.e.
$$
\overline{\mu}^n=\frac1n\sum_{i=1}^n\delta_{\tau^i},
$$
where we use the notation $\delta_x$ for the probability measure putting mass $1$ on the singleton $\{x\}$.
Notice that $\overline{\mu}^n[0,t)$ represents the proportion of depositors who tried to withdraw before time $t$, and that the time of endogenous default is given by 
$$
\tau^{endo}=\inf \{t\in (0,T);\; \overline{\mu}^n[0,t)>L(B_t)\},\footnote{Here and throughout the text we write $\overline{\mu}^n[0,t)$ in place of the somewhat more precise $\overline{\mu}^n([0,t))$.}
$$
with the convention that the infimum of the empty set is defined as $T$. 
For the sake of simplicity we assume that once a depositor runs, he cannot get back in the game, in other words, his decision is irreversible.

\subsubsection*{\textbf{Depositor Strategic Behavior}}
We now explain the strategic behavior of the $n$ depositors. We denote by $\FF^i=(\F^i_t)_{t\ge 0}$ the information available to player $i\in\{1,\cdots,n\}$. This is a filtration, $\F^i_t$ representing the information available to player $i$ at time $t$. 
In the particular case which we discuss first, these filtrations are all identical. They are based on a perfect (though non-anticipative) observation of the signal $(B_t)_{0\le t\le T}$. We call this situation \emph{public monitoring}.  In a more realistic form of the model, the filtration $\FF^i$ is generated by the process $X^{i,n}=(X^{i,n}_t)_{t \ge 0}$ and the process $(\overline{\mu}^n[0,t])_{t\ge 0}$,  where $X^{i,n}_t$ is the \emph{private signal} of depositor $i$, namely the value of the observation of $B_t$ he or she can secure at time $t$. We shall assume that it is of the form
$$
X^{i,n}_t=B_t+\sigma W^i_t
$$
where $\sigma>0$ and for $i\in\{1,\cdots, n\}$, the processes $(W^i_t)_{t\ge 0}$ are independent identically distributed (i.i.d.) stochastic processes  (also independent of $B$) representing idiosyncratic noise terms blurring the observations of the exact value $B_t$ of the assets of the bank. When $\FF^i$ is generated by $X^{i,n}$ and $(\overline{\mu}^n[0,t])_{t\ge 0}$, we talk about \emph{private monitoring} of the asset value of the bank. However, for an even more realistic form of the model, we shall require that the filtration $\FF^i$  is generated simply by $X^{i,n}$ and does not include the information provided by the process $(\overline{\mu}^n[0,t])_{t\ge 0}$, which incorporates the private signals of the other depositors. This model should be more challenging mathematically as the individual depositors will have to choose their withdrawal strategies in a distributed manner, using only the information contained in their private signals, ignoring the process $(\overline{\mu}^n[0,t])_{t\ge 0}$.

In any case, the filtrations $\FF^i$ will be specified in each particular application and will play the following role: 
the time $\tau^i$ chosen by agent $i$ is required to be a $\FF^i$-stopping time in order to be admissible.

Given that all the other players $j\ne i$ have chosen their times $\tau^j$ to try to withdraw their deposits, the payoff $P^i(\tau^{-i},\tau^i)$ to depositor $i$ 
for trying to run on the bank at time $\tau^i$ can be written (recalling that $D^i_0=1/n$) as 
$$
P^i(\tau^{-i},\tau^i) = D_0^i\wedge \bigg(L(B_{\tau^i})-\overline{\mu}^n[0,\tau^i)\bigg)^+ = D_0^i\wedge \bigg(L(B_{\tau^i})-\frac{1}{n}\sum_{k=1}^n 1_{[0,\tau^i)}(\tau^k)\bigg)^+
$$
and the problem of depositor $i$ is then to choose for $\tau^i$, the $\FF^i$-stopping time solving the maximization problem
$$
J^i(\tau^{-i})=\sup_{\tau^i}\E\bigg[e^{(\overline{r}-r)\tau^i}P^i(\tau^{-i},\tau^i)\bigg]
$$
which is an optimal stopping problem. 
Any solution $\tau^i$ of this maximization problem represents a best response of player $i$ to the choices $\tau^{-i}$ of the other depositors.
Finding a set of stopping times $\tau^i$ for $i=1,\ldots,n$ satisfactory to all the players simultaneously is essentially finding a fixed point to the search for best responses. This is achieved by finding a Nash equilibrium for this game.

\subsubsection*{\textbf{Solution in the Case of Public Monitoring through Perfect Observation}}
If we assume that $\sigma=0$, in which case $\FF^i=\FF^B=(\F^B_t)_{t\ge0}$, at time $t$ each depositor knows the past up to time $t$ of the asset value $B_s$ for $s\le t$, and if all the depositor decisions (to run or not to run) are based only on this information, then for each $t\in[0,T]$, 
$\overline{\mu}^{n}[0,t]\in\F^B_t$ since this information is known by all the depositors  at time $t$. 

\begin{proposition}
In the case of public information, if we define the stopping time $\hat\tau$ by
$$
\hat\tau=T\wedge \inf\{t>0;\, L(B_t)\le 1\},
$$
then the unique Nash equilibrium is when all the depositors decide to run at time $\hat\tau$.
\end{proposition}
So a bank run occurs as soon as the bank has liquidity problems, even if this is long before it is insolvent. Notice also that according to this proposition, all the depositors experience full recovery of their deposits, which is in flagrant contrast with typical bank runs in which most depositors usually experience significant losses.
\proof{
We first argue that we have indeed identified a Nash equilibrium. If all the other depositors but $i$ choose the strategy given by the running time $\hat\tau$,
we show that player $i$ cannot do better than choosing to also run at time $\hat\tau$. If $L(B_{\hat{\tau}})\le 1$, all the others depositors run immediately, and the only hope investor $i$ has to get something out of his or her deposit is to run at time $\hat{\tau}$ as well. Now if $L(B_{\hat{\tau}})>1$, no depositor has a reason to run while $L(B_t)>1$ since by not running for a small time interval while $L(B_t)$ is still strictly greater than $1$, he or she can earn the superior interest $\overline{r}>r$ without facing any risk. This proves that every depositor using $\hat\tau$ as time to run is a Nash equilibrium. 
We do not give the proof of the fact that this Nash equilibrium is the unique Nash equilibrium since we are not really interested in the public information case. 
}

\subsubsection*{\textbf{The Mean Field Game Formulation}}
We now consider an asymptotic regime corresponding to a large number of depositors, sending $n \rightarrow \infty$, and we track the behavior of a representative depositor with initial deposit $D_0 > 0$. Although the payoffs $P^i$ themselves decrease to zero, as $D^i_0 = 1/n$,  we are not terribly concerned with the asymptotic behavior of the \emph{values} of the objective functions, so we may simply rescale $P^i$ to $nP^i$ in the $n$-player game without altering the set of equilibria. Indeed, the main quantity we wish to control in this asymptotic regime is the empirical distribution of the equilibrium stopping times, as this contains all of the information describing the timing of the bank run.

When $n$ is large, the usual heuristics for mean field games suggest that, if the process $B$ giving the asset value of the bank is not deterministic, $\overline{\mu}^n$ approaches a random measure $\mu$. In particular, this limiting $\mu$ should depend on the time-evolution of $B$ in the sense that $\mu[0,t]$ should be $\F^B_t$-measurable for each $t \in [0,T]$. If such a probability measure $\mu$ is fixed, one defines the individual payoff $P^\mu(t, y)$ of a withdrawal attempt at time $t$ when the value of the assets of the bank is $y$ as:
$$
P^\mu(t, y)=D_0\wedge\bigg( L(y)-\mu[0,t]\bigg)^+,
$$
and the optimal time for a representative depositor to claim his or her deposit back will be given by the stopping times
solving the optimal stopping problem:
$$
\sup_{0\le \theta\le T}\E[  e^{(\overline{r} - r)\theta}P^\mu(\theta,B_\theta) ].
$$
The above maximization is understood over all the $\FF^X$-stopping times $\theta$ where the filtration $\FF^X=(\F_t^X)_{0\le t\le T}$ is the filtration generated by the signal $X_t=B_t+\sigma W_t$ observed at time $t$ by our generic investor. 
Here $(W_t)_{0\le t\le T}$ is a process independent of $(B_t)_{0\le t\le T}$ and sharing the same distribution as each of the $W^i$ from before. If we can solve this optimal stopping problem for each (random) measure $\mu$, we can define a map $\mu \to \text{Law}(\hat\tau |B)$ where $\hat\tau$ is an optimal stopping time, and the final step of the mean field game approach is to find a fixed point for this map. The following section formulates the mean field game more precisely and explains the connection with the $n$-player game.

\section{General Mean Field Games of Timing: Main Results}
\label{se:results}

A compact set of times $\T \subset [0,\infty]$ is fixed throughout, which we assume is either discrete or of the form $[0,T]$ for some $T\in [0,\infty]$. 
Fix two filtered probability spaces $(\Omega^{\mathrm{com}},\F^{\mathrm{com}},\FF^{\mathrm{com}},\PP^{\mathrm{com}})$ and $(\Omega^{\mathrm{ind}},\F^{\mathrm{ind}},\FF^{\mathrm{ind}},\PP^{\mathrm{ind}})$, which will house a \emph{common noise} and an \emph{independent} (or \emph{idiosyncratic}) \emph{noise}, respectively. We are given also a filtration $\FF^{\mathrm{sig}}=(\F^{\text{sig}}_t)_{t \in \T}$ on the product space $\Omega^{\mathrm{com}} \times \Omega^{\mathrm{ind}}$ with $\F^\text{sig}_t \subset \F^{\mathrm{com}}_t \otimes \F^{\mathrm{ind}}_t$ for every $t$. This filtration represents the \emph{signal} or information available to an agent. Rather than observing the full filtration $\FF^{\mathrm{com}} \otimes \FF^{\mathrm{ind}}$ of the underlying noises, an agent sees only $\FF^{\mathrm{sig}}$.
An \emph{objective function} is given,
\[
F : \Omega^{\mathrm{com}} \times \Omega^{\mathrm{ind}} \times \P(\T) \times \T \rightarrow \R,
\]
where $\P(\T)$ denotes the set of Borel probability measures on $\T$.
Here $F(\omega^0,\omega^1,m,t)$ represents the reward an agent achieves by stopping at time $t$, given the values $(\omega^0,\omega^1)$ of the common and independent noises, and given the distribution $m$ of other agents' stopping times.

With these ingredients, we will formulate both an $n$-player game and its continuum limit as $n\rightarrow\infty$. Assumption \textbf{A} below will clarify the precise assumptions (measurability, continuity, etc.) on $F$, and until then we will tacitly assume the expectations make sense.

\begin{example} \label{ex:bankruns}
In the example presented in Section \ref{se:runs}, we make the following identifications. Let $\T = [0,T]$, and let $(\Omega^{\mathrm{com}},\F^{\mathrm{com}},\FF^{\mathrm{com}},\PP^{\mathrm{com}})$ and $(\Omega^{\mathrm{ind}},\F^{\mathrm{ind}},\FF^{\mathrm{ind}},\PP^{\mathrm{ind}})$ both equal the Wiener space of continuous real-valued paths. That is, $\Omega^{\mathrm{com}} = \Omega^{\mathrm{ind}} = C([0,T])$ is equipped with the Borel $\sigma$-field, the Wiener measure, and the natural (augmented) filtration. The sub-filtration $\FF^{\mathrm{sig}}$ is the complete filtration generated by the process $(W^0_t + \sigma W^1_t)_{t \in [0,T]}$, where $W^0$ and $W^1$ denote the projections from $\Omega^{\mathrm{com}} \times \Omega^{\mathrm{ind}}$ to $\Omega^{\mathrm{com}}$ and to $\Omega^{\mathrm{ind}}$, respectively.
The objective function (after a renormalization) is
\[
F(\omega^0,\omega^1,m,t) =e^{(\overline{r} - r)t}\left[1\wedge \bigg( L(\omega^0_t)-m[0,t]\bigg)^+\right].
\]
Note that in this example the independent noise $\omega^1$ does not appear in the payoff, and its only role is in specifying the information structure $\FF^{\mathrm{sig}}$.
\end{example}

\subsection{The $n$-player game} \label{se:nplayergame}
The $n$-player game for $n \ge 1$ is defined on the product space
\[
(\overline{\Omega},\overline{\F},\overline{\FF},\overline{\PP}) := (\Omega^{\mathrm{com}},\F^{\mathrm{com}},\FF^{\mathrm{com}},\PP^{\mathrm{com}}) \otimes \bigotimes_{k=1}^\infty(\Omega^{\mathrm{ind}},\F^{\mathrm{ind}},\FF^{\mathrm{ind}},\PP^{\mathrm{ind}}).
\]
A typical element of $\overline{\Omega}$ is denoted $\vec{\omega} = (\omega^0,\omega^1,\ldots)$, with $\omega^0 \in \Omega^{\mathrm{com}}$ and $\omega^i \in \Omega^{\mathrm{ind}}$ for $i\ge 1$. We call $\omega^0$ the \emph{common noise} and $\omega^i$ the \emph{idiosyncratic noise} of agent $i$. Define the projections 
\[
W^i(\omega^0,\omega^1,\ldots) = \omega^i,
\]
for $i=0,1,\ldots$. Finally, for $i \ge 1$, define the filtration
$\FF^i=(\F^i_t)_{t \in \T}$ of the $i^\text{th}$ agent by 
\[
\F^i_t := (W^0,W^i)^{-1}(\F^{\text{sig}}_t) := \sigma\{\{(W^0,W^i) \in C\} : C \in \F^{\text{sig}}_t\}.
\]

Define the empirical measure map $\overline{\mu}^n : \T^n \rightarrow \P(\T)$ by
\begin{align}
\overline{\mu}^n(t_1,\ldots,t_n) = \frac{1}{n}\sum_{k=1}^n\delta_{t_k}. \label{def:empirical-measure}
\end{align}
We will make use of the following common notation: given $\vec{e} = (e_1,\ldots,e_n) \in E^n$ for some set $E$, define
\begin{align*}
\vec{e}^{\,-i} = (e_1,\ldots,e_{i-1},e_{i+1},\ldots,e_n), \quad\text{and}\quad (\vec{e}^{\,-i},x) = (e_1,\ldots,e_{i-1},x,e_{i+1},\ldots,e_n),
\end{align*}
for $x \in E$ and $i=1,\ldots,n$.
To minimize the number of parentheses, we abuse notation somewhat by writing $\overline{\mu}^n(\vec{t}^{-i},s)$ in lieu of $\overline{\mu}^n((\vec{t}^{-i},s))$, when $\vec{t} \in \T^n$ and $s \in \T$.
For $\epsilon \ge 0$, we say that $\vec{\tau}=(\tau_1,\ldots,\tau_n)$ is an \emph{$\epsilon$-Nash equilibrium} if $\tau_i$ is an $\FF^i$-stopping time (defined on $\overline{\Omega}$) and 
\begin{align*}
\E\left[F\left(W^0,W^i,\overline{\mu}^n(\vec{\tau}),\tau_i\right)\right] \ge \E\left[F\left(W^0,W^i,\overline{\mu}^n\left(\vec{\tau}^{\,-i},\sigma\right),\sigma\right)\right] - \epsilon,
\end{align*}
for every alternative $\FF^i$-stopping time $\sigma$, for each $i=1,\ldots,n$.

\subsection{The mean field game}
We next define the infinite-agent counterpart of the above game, called the \emph{mean field game}, which is formulated on the product space $\Omega^{\mathrm{com}}\times\Omega^{\mathrm{ind}}$. We write $\E$ for expectation under the product measure $\PP^{\mathrm{com}} \times \PP^{\mathrm{ind}}$, and we write $W^0$ and $W^1$ for the projections onto $\Omega^{\mathrm{com}}$ and $\Omega^{\mathrm{ind}}$, respectively.

\begin{definition} \label{def:MFE}
A \emph{strong mean field equilibrium (MFE)} is a $\FF^{\mathrm{sig}}$-stopping time $\tau^*$ on $\Omega^{\mathrm{com}}\times\Omega^{\mathrm{ind}}$ satisfying
\begin{align*}
\E\left[F(W^0,W^1,\mu,\tau^*)\right] \ge \E\left[F(W^0,W^1,\mu,\tau)\right],
\end{align*}
for every alternative $\FF^{\mathrm{sig}}$-stopping time $\tau$, where
\begin{align}
\mu = \PP^{\mathrm{com}} \times \PP^{\mathrm{ind}}[\tau^*\in \cdot \, | \, W^0] \label{def:MFE-law}
\end{align}
is the regular conditional law of $\tau^*$ given $W^0$.
\end{definition}

We say \emph{strong MFE} here because later, in Definition \ref{def:weakMFE}, we will later introduce a notion of \emph{weak MFE}. 
One justification of this strong equilibrium concept is the following theorem, which explains how to use a strong MFE to construct approximate Nash equilibria for the $n$-player games. First, some assumptions are needed.
In the following, consider the topology $\sigma(\P(\T),B(\T))$ generated by the set $B(\T)$ of bounded measurable functions of $\T$; that is, $\sigma(\P(\T),B(\T))$ is the coarsest topology on $\P(\T)$ such that the map $m \mapsto \int_\T\varphi\,dm$ is continuous for every $\varphi \in B(\T)$.
Define the total variation of a signed measure $\nu$ on $\T$ by
\[
\|\nu\|_{TV} = \sup\left\{\int_\T f\,d\nu : f \in B(\T), \ \sup_{t \in \T}|f(t)| \le 1\right\}.
\]
In the following, $\Omega^{\mathrm{com}} \times \Omega^{\mathrm{ind}}$ is always equipped with the probability measure $\PP^{\mathrm{com}} \times \PP^{\mathrm{ind}}$.

\begin{assumptionA*} \label{assumption:A}
$\ $
\begin{enumerate}
\item[(A.1)] $F$ is jointly measurable, with $\P(\T)$ equipped with the $\sigma$-field generated by the maps $m \mapsto m(C)$, where $C \subset \T$ is a Borel set.\footnote{This $\sigma$-field agrees with the Borel $\sigma$-field generated by the topology of weak convergence on $\P(\T)$.
}
\item[(A.2)] For almost every $(\omega^0,\omega^1)$, the map $m \mapsto F(\omega^0,\omega^1,m,t)$ is $\sigma(\P(\T),B(\T))$-continuous, uniformly in $t$. That is, for each $m_0 \in \P(\T)$, the map
\[
m \mapsto \sup_{t \in \T}\left|F(\omega^0,\omega^1,m,t) - F(\omega^0,\omega^1,m_0,t)\right|
\]
is $\sigma(\P(\T),B(\T))$-continuous at $m_0$.
\item[(A.3)] It holds that
\begin{align}
\E\left[\sup_{m \in \P(\T)}\sup_{t \in \T}\left|F(W^0,W^1,m,t)\right|\right] < \infty. \label{def:A.5}
\end{align}
\end{enumerate}
\end{assumptionA*}

Assumption (A.2) may appear difficult to verify. On the contrary, there are two broad classes of examples it covers. First, because $\sigma(\P(\T),B(\T))$ is finer than the topology of weak convergence, replacing $\sigma(\P(\T),B(\T))$-continuity with weak continuity (i.e., continuity with respect to the topology of weak convergence) is enough. Moreover, because $[0,\infty]$ is compact, joint continuity of $F(\omega^0,\omega^1,m,t)$ in $(m,t)$ implies (A.2). The second class of examples, and indeed the one that motivates the use of the topology $\sigma(\P(\T),B(\T))$, consists of functions $F$ of the form $F(\omega^0,\omega^1,m,t) = G(\omega^0,\omega^1,m[0,t],t)$, where $G : \Omega^{\mathrm{com}} \times \Omega^{\mathrm{ind}} \times [0,1] \times \T \rightarrow \R$ is measurable. If $G=G(\omega^0,\omega^1,u,t)$ is continuous in $u$, uniformly in $t$, for each fixed $(\omega^0,\omega^1)$, then $F$ satisfies (A.2). This follows from a simple lemma, proven in Section \ref{se:proofs-supermodular}.

\begin{lemma} \label{le:glivenko-cantelli}
For each $m_0 \in \P(\T)$, the map
\[
m \mapsto \sup_{t \in \T}|m[0,t] - m_0[0,t]|
\]
is $\sigma(\P(\T),B(\T))$-continuous at $m_0$.
\end{lemma}

The following main result illustrates how a mean field equilibrium can be used to construct near-equilibria for the $n$-player games. Its proof is given in Section \ref{se:proof-converse}.

\begin{theorem} \label{th:main-converse}
Suppose assumption \textbf{A} holds.
Suppose $\tau^*$ is a mean field equilibrium, and let $\mu$ be as in \eqref{def:MFE-law}. For each $k$ define an $\FF^k$-stopping time on $\overline{\Omega}$ by
\[
\tau^k(\omega^0,\omega^1,\ldots,\omega^n) = \tau^*(\omega^0,\omega^k).
\]
Then there exist $\epsilon_n \ge 0$ with $\epsilon_n \rightarrow 0$ such that $\vec{\tau}^{\,n} = (\tau^1,\ldots,\tau^n)$ is an $\epsilon_n$-Nash equilibrium for each $n$, and moreover 
\begin{align*}
\lim_{n\rightarrow\infty}\E\left[F(W^0,W^k,\overline{\mu}^n(\vec{\tau}^{\,n}) ,\tau^k)\right] &= \E\left[F(W^0,W^1,\mu,\tau^*)\right], \text{ for each } k.
\end{align*}
\end{theorem}

\subsection{Strategic Complementarities and Existence of MFEs}

An existence result for strong MFE is available, even for discontinuous $F$, as long as a suitable \emph{complementarity} property holds, as was mentioned in the introduction. This section draws heavily on ideas from literature on games with strategic complementarities \cite{milgrom1990rationalizability,vives1990nash}, which is abundant with existence proofs based more on monotonicity than continuity. In the following, let us say that a $\P(\T)$-valued random variable $\mu$ on $\Omega^{\mathrm{com}}\times\Omega^{\mathrm{ind}}$ is an \emph{$\FF^{\mathrm{com}}$-adapted random measure} if $\mu[0,t]$ is $\F^{\mathrm{com}}_t$-measurable for every $t \in \T$.

\begin{assumptionB*}
$\ $
\begin{enumerate}
\item[(B.1)] $\FF^{\mathrm{sig}}$ is right-continuous.
\item[(B.2)] For every pair of $\FF^{\mathrm{com}}$-adapted random measures $\mu,\tilde{\mu}$ satisfying $\mu[0,t]\ge \tilde{\mu}[0,t]$ for all $t\in\T$ a.s., the process $(M_t)_{t \in \T}$ defined by
\[
M_t=F(W^0,W^1,\tilde{\mu},t) - F(W^0,W^1,\mu,t) 
\]
is a submartingale. 
\item[(B.3)] For each $m \in \P(\T)$, $t \mapsto F(W^0,W^1,m, t)$ is upper semicontinuous, almost surely.
\item[(B.4)] Conditions (A.1) and (A.3) hold.
\end{enumerate}
\end{assumptionB*}

If $\mu \le \tilde{\mu}$ in the sense of stochastic order (i.e., if $\tilde{\mu}[0,t]\le \mu[0,t]$ a.s. for each $t\in[0,T]$), and if $\tau \le \tilde{\tau}$ are stopping times, taking expectations in the submartingale property of $M_t$ in assumption (B.2) yields
\begin{equation}
\label{fo:ID}
\E[F(W^0,W^1,\tilde{\mu},\tilde{\tau})] - \E[F(W^0,W^1,\tilde{\mu},\tau)]\ge  \E[F(W^0,W^1,\mu,\tilde{\tau})] - \E[F(W^0,W^1,\mu,\tau)], 
\end{equation}
property which is usually called \emph{increasing differences}.
Intuitively, assumption (B.2) requires that for ``larger'' $\mu$ the function $F$ increases more rapidly in expectation with $t$ than it does for smaller $\mu$. These hypotheses introduce strategic complementarities in the game and recast the game of timing model as a supermodular game.
They are natural in the context of bank run models, in which the measure $\mu$ captures how early people run to the bank. Indeed, if $\tilde{\mu} \ge \mu$ in stochastic order, then under $\mu$ more people have run to the bank earlier. Under $\mu$, the reward an agent gains by waiting from $\tau$ to $\tilde{\tau} > \tau$ should not exceed the same reward under $\tilde{\mu}$. In other words, if people tend to run to the bank earlier, the ``cost of waiting'' for an investor should be greater.

While assumption \textbf{B} is all that is needed for existence, the following stronger assumption will enable a better understanding of the structure of equilibria.

\begin{assumptionC*}
$\ $
\begin{enumerate}
\item[(C.1)] $F(W^0,W^1,m,t)$ is almost surely jointly continuous in $(m,t)$, when $\P(\T)$ is endowed with the topology of weak convergence.
\item[(C.2)] Condition (A.3) holds.
\end{enumerate}
\end{assumptionC*}

\begin{theorem} \label{th:existence-B}
If assumption \textbf{B} holds, then there exists a strong MFE. If both assumptions \textbf{B} and \textbf{C} hold, then there exist strong MFEs $\tau^*$ and $\theta^*$ such that for any strong MFE $\tau$ we have $\theta^*\le \tau\le \tau^*$ a.s.  
\end{theorem}

Some examples of assumption \textbf{B} are as follows. First,  suppose that for every $t\le t'$, every $(\omega^0,\omega^1)$, and every $m,m' \in \P(\T)$ satisfying $m \le m'$ in stochastic order (meaning $m[0,s] \ge m'[0,s]$ for every $s$), we have
\begin{align*}
F(\omega^0,\omega^1,m',t') - F(\omega^0,\omega^1,m',t) \ge F(\omega^0,\omega^1,m,t') - F(\omega^0,\omega^1,m,t).
\end{align*}
Then the submartingale part of assumption \textbf{B} holds trivially, as the process $M$ is nondecreasing.
The following proposition and remark show how to verify assumption \textbf{B} for a large class of examples based on diffusion processes.

\begin{proposition} \label{pr:diffusion}
Suppose $\T = [0,T]$, and assume the space $\Omega^{\mathrm{com}} \times \Omega^{\mathrm{ind}}$ supports a continuous It\^o diffusion $X=(X_t)_{t \in [0,T]}$ with infinitesimal generator $\L$ defined on all smooth functions $\varphi$ of compact support by
\begin{align*}
\L  f(x) = b(x)\cdot\nabla f(x) + \frac{1}{2}\mathrm{Tr}[a(x)\nabla^2f(x)],
\end{align*}
where $b$ and $a$ are measurable functions with values in $\R^d$ and the set of positive semidefinite $d \times d$ matrices, respectively.
Assume $F$ is of the form $F(\omega^0,\omega^1,m,t)=f(X_t(\omega^0,\omega^1),\varphi*m(t),t)$, where $f$ is bounded and
\[
\varphi*m(t) = \int_{[0,T]}\varphi(t-s)\,m(ds).
\]
Moreover, assume $f : \R^d \times \R \times [-T,T] \ni (x,y,t)\mapsto f(x,y,t)\in \R$ has two bounded continuous derivatives in $x$ and one in both $y$ and $t$, and $\varphi : [0,T] \rightarrow \R$ is continuous.
Suppose one of the following holds:
\begin{enumerate}[(i)]
\item $\varphi$ is nondecreasing and convex, $\partial_yf \ge 0$, and also $\L_xf + \partial_tf$ and $\partial_yf$ are nondecreasing in $y$ for each fixed $(x,t)$, where $\L_x$ denotes the action of $\L$ on the $x$ variable.
\item $\varphi$ is nonincreasing and convex, $\partial_yf \le 0$, and also $\L_xf + \partial_tf$ and $\partial_yf$ are nonincreasing in $y$ for each fixed $(x,t)$.
\end{enumerate}
Assume finally that
\begin{align}
\E\int_0^T\sup_{y \in \R}\left|a(X_t)\nabla_xf(X_t,y,t)\right|^2dt < \infty. \label{def:ito-integrability}
\end{align}
Then assumptions (B.2-3) and (C.1) hold.
\end{proposition}
\begin{proof}
The only nontrivial claim is that the submartingale property (B.2) holds. To check this, fix two $\FF^{\mathrm{com}}$-adapted random measures $\mu$ and $\widetilde{\mu}$ satisfying $\mu[0,t] \ge \tilde{\mu}[0,t]$ a.s. for every $t$. By It\^o's formula,
\begin{align*}
df(X_t,\varphi*\mu(t),t) &= \left\{\L_xf(X_t,\varphi*\mu(t),t) + \partial_tf(X_t,\varphi*\mu(t),t) + \partial_yf(X_t,\varphi*\mu(t),t)\,\varphi'*\mu(t)\right\}dt \\
	&\quad\quad\quad +  \nabla_xf(X_t,\varphi*\mu(t),t) \cdot a(X_t)dB_t,
\end{align*}
where $B=(B_t)_{t \in [0,T]}$ is a standard Brownian motion (defined perhaps on an extension of the probability space).
The assumption \eqref{def:ito-integrability} implies that the $dB_t$ term is a martingale. To  show that $f(B_t,\varphi*\mu(t),t) - f(B_t,\varphi*\widetilde{\mu}(t),t)$ is a submartingale, it suffices to check that its $dt$ term is always nonnegative.
If $\mu \le \widetilde{\mu}$ are as in assumption \textbf{B}, then the $dt$ term of $df(B_t,\varphi*\mu(t),t) - df(B_t,\varphi*\widetilde{\mu}(t),t)$ is precisely
\begin{align*}
&\L_xf(B_t,\varphi*\widetilde{\mu}(t),t) - \L_xf(B_t,\varphi*\mu(t),t) \\
+ &\partial_tf(B_t,\varphi*\widetilde{\mu}(t),t) - \partial_tf(B_t,\varphi*\mu(t),t) \\
+ &\partial_yf(B_t,\varphi*\widetilde{\mu}(t),t)\,\varphi'*\widetilde{\mu}(t) - \partial_yf(B_t,\varphi*\mu(t),t)\,\varphi'*\mu(t).
\end{align*}
Now note that if $m \le \tilde{m}$ in stochastic order then $\int g\,dm \le \int g\,d\tilde{m}$ for every nondecreasing function $g$, and in particular if $\varphi$ is nondecreasing (resp. nonincreasing) and convex then $\varphi * \tilde{m} \ge \varphi * m$ (resp. $\le$) and $\varphi' * \tilde{m} \ge \varphi' * m$ pointwise. With this in mind, it is straighforward to check that either set of assumptions ensures that the above quantity is nonnegative.
\end{proof}

The assumption \eqref{def:ito-integrability} is not very restrictive; it holds as soon as $\nabla_xf$ and $a$ are bounded, or more generally under linear growth assumptions and suitable integrability for the initial state $X_0$.
The conditions (i-ii) are more restrictive, and
the following simple result illustrates more broadly the limitations of assumption (B.2) in handling a very natural form of mean field interaction.
In particular, Proposition \ref{pr:impossibility} suggests that our bank run model cannot satisfy assumption (B.2) because of the dependence of $F(\omega^0,\omega^1,m,t)$ on $m[0,t]$.

\begin{proposition} \label{pr:impossibility}
Suppose $\T$ is continuous, and suppose $F(\omega^0,\omega^1,m,t) = G(m[0,t])$ for some continuous $G : [0,1] \rightarrow \R$ which we assume is differentiable on $(0,1)$. If $F$ satisfies assumption (B.2), then $G$ is constant.
\end{proposition}
\begin{proof}
For $m,\widetilde{m} \in \P([0,T])$ with $m \le \widetilde{m}$, assumption (B.2) implies that the deterministic process $G(\widetilde{m}[0,t]) - G(m[0,t])$ is a submartingale, which means simply that it is nondecreasing. In other words, for $0 \le s \le t \le T$,
\begin{align*}
G(\widetilde{m}[0,t]) - G(\widetilde{m}[0,s]) \ge G(m[0,t]) - G(m[0,s]).
\end{align*}
Dividing by $t-s$ and taking limits, we find
\[
G'(F_2(t))f_2(t) \ge G'(F_1(t))f_1(t),
\]
assuming $F_1(t)=m[0,t]$ and $F_2(t)=\widetilde{m}[0,t]$ have derivatives $f_1$ and $f_2$. The point is that stochastic dominance is not sensitive to changes in density. Given $u \in (0,1)$, there exist $m \le \widetilde{m}$ and $t \in [0,T]$ such that $F_2(t)=u$ while $f_1(t)=0$ and $f_2(t)=1$, which implies $G'(u) \ge 0$. On the other hand, given $u \in (0,1)$, there exist $m \le \widetilde{m}$ and $t \in [0,T]$ such that $F_1(t)=u$ while $f_1(t)=1$ and $f_2(t)=0$, which implies $G'(u) \le 0$. Thus $G' \equiv 0$ on $(0,1)$.
\end{proof}


\subsection{An example} \label{se:nutz}
The recent model of Nutz \cite{Nutz}, or at least many specializations thereof, can be shown to satisfy our assumption \textbf{A}. The explicit computations of equilibria in \cite{Nutz} can be used in tandem with Theorem \ref{th:main-converse} to construct $n$-player approximate equilibria.

We describe only a simple case of this model, from \cite[Section 5.1]{Nutz}. Let $\T = [0,\infty]$, and suppose $\Omega^{\mathrm{com}} = \Omega^{\mathrm{ind}} = D_{\uparrow}$ is the space of nondecreasing right-continuous real-valued functions on $[0,\infty)$. Note that for any $f \in D_{\uparrow}$ the limit $f(\infty) = \lim_{t \rightarrow\infty}f(t)$ exists in $\R \cup \{\infty\}$. Constants $c > 0$ and $r > 0$ are given, and the objective function is
\begin{align*}
F(\omega^0,\omega^1,m,t) = \exp\left(\int_0^t\left(r - \omega^0(s) - \omega^1(s) - cm[0,s]\right)ds\right).
\end{align*}
The process $\gamma_t(\omega^0,\omega^1) = \omega^0(t) - \omega^1(t) - cm[0,t]$ can be interpreted as the agent's perception of the rate of bank failure. This perceived rate changes over time, depending on a common factor $\omega^0$ and an independent factor $\omega^1$, as well as the fraction of agents who have already run to the bank.
In fact, this is not the primitive form of the objective function given in \cite{Nutz} but is instead derived in Lemma 2.1 therein (more precisely, equation (2.3)).

It is straightforward to check that assumptions (A.1) holds for this example, and we may use Lemma \ref{le:glivenko-cantelli} to check that (A.2) holds as well. Assumption (A.3) holds as long as
\[
\E\left[\exp\left(\sup_{t \ge 0}\int_0^t\left(r - W^0(s) - W^1(s)\right)ds\right)\right] < \infty.
\]
On the other hand, it appears that assumption (B.2) fails for this class of models in most cases. However, the arguments of \cite{Nutz} lead to explicit computations of MFE when agents have access to enough information, namely when $W^0$ and $W^0 + W^1$ are both adapted to $\FF^{\mathrm{sig}}$. Theorem \ref{th:main-converse} can then be used to construct explicit $n$-player approximate equilibria.

\section{Beyond complementarities: Weak equilibria}
\label{se:weakEQ}

This section explains how to move past the restrictive assumptions of complementarities by deriving an existence result and a limit theorem under the modest continuity assumptions on the objective function.
Our time set is now a finite interval $\T = [0,T]$, $T > 0$. 
Let $\C = C([0,T])$ denote the space of continuous real-valued functions on $[0,T]$, endowed with the supremum norm. For a Polish space $E$, we always write $\P(E)$ for the space of Borel probability measures on $E$, endowed with the topology of weak convergence.

For the rest of this section we specify
\[
\Omega^{\mathrm{com}} = \Omega^{\mathrm{ind}} = \C,
\]
The common noise and independent noise will now both be one-dimensional standard Brownian motions, for the sake of simplicity. This could be generalized in various directions, most obviously by making these Brownian motions multi-dimensional, and this would not alter the analysis.
Let us write $\W$ for the Wiener measure on $\C$, and specialize the setup of Section \ref{se:results} by setting
\[
\PP^{\mathrm{com}} = \PP^{\mathrm{ind}} = \W.
\]
Write $B=(B_t)_{t \in [0,T]}$ and $W=(W_t)_{t \in [0,T]}$ for the canonical processes on $\C^2$, and let $\FF^B=(\F^B_t)_{t \in [0,T]}$ and $\FF^W=(\F^W_t)_{t \in [0,T]}$ denote their natural (raw) filtrations.
The objective function is now a function $F : \C^2 \times \P([0,T])\times [0,T] \rightarrow \R$.
Note that the full information version of the bank run model of Section \ref{se:runs} fits into this specialized setup; see Example \ref{ex:bankruns}.

The equilibrium concept for the $n$-player game is as in Section \ref{se:nplayergame}, but now with full information: Given independent Wiener processes $B$ and $(W^i)_{i=1}^n$, agent $i$ chooses a random time $\tau^i$, which is required to be a stopping time relative to the full filtration $\FF^{B,W^1,\ldots,W^n}$ generated by $(B,W^1,\ldots,W^n)$, but we will not spell out the details.
Recall that for $\epsilon \ge 0$ we say that $\vec{\tau}=(\tau_1,\ldots,\tau_n)$ is an \emph{$\epsilon$-Nash equilibrium} if $\tau_i$ is an $\FF^{B,W^1,\ldots,W^n}$-stopping time (with values in $[0,T]$) and if 
\begin{align*}
\E\left[F\left(B,W^i,\overline{\mu}^n(\vec{\tau}),\tau_i\right)\right] \ge \E\left[F\left(B,W^i,\overline{\mu}^n\left(\vec{\tau}^{\,-i},\sigma\right),\sigma\right)\right] - \epsilon, 
\end{align*}
for every alternative $\FF^{B,W^1,\ldots,W^n}$-stopping time $\sigma$, for each $i=1,\ldots,n$. Unfortunately, our proof techniques seem to be restricted to this full information case; an earlier version of this paper contained partial-information analogs of the following results, but there was a flaw in the proof.

We are interested in describing the limiting behavior of Nash equilibria, as $n\rightarrow\infty$, in addition to the converse construction of Theorem \ref{th:converselimit}. To this end, we introduce notions of strong and weak equilibria in analogy with strong and weak solutions of stochastic differential equations. The strong equilibrium is exactly as in Definition \ref{def:MFE}, but with full information:

\begin{definition} \label{def:strongMFE}
A \emph{strong mean field equilibrium (MFE)} is a $\FF^{B,W}$-stopping time $\tau^*$ defined on $\C^2$, equipped with the Wiener measure $\W^2$, satisfying
\begin{align*}
\E\left[F(B,W,\mu,\tau^*)\right] \ge \E\left[F(B,W,\mu,\tau)\right],
\end{align*}
for every $\FF^{B,W}$-stopping time $\tau$, where $\mu = \W^2[\tau^*\in \cdot | B]$ is the conditional law of $\tau^*$ given $B$.
\end{definition}

The definition of a weak MFE requires care. Because we will work heavily with weak limits, we must prepare for some loss of measurability, in light of the following basic fact of weak convergence:  if $(Z,Y_n)$ are random variables converging weakly to $(Z,Y)$, and if $Y_n$ is $Z$-measurable for each $n$, then there is absolutely no reason to expect that $Y$ is $Z$-measurable in the limit, despite the fact that $Z$ does not depend on $n$. For this reason, we define a notion of weak MFE in which $\mu$ is not required to be $B$-measurable, and $\tau$ may be a \emph{randomized stopping time}, in a sense made precise below. In analogy with the definition of weak solutions for stochastic differential equations, we base the definition of weak solution on the properties of the joint distribution of $(B,W,\mu,\tau)$. In fact, for reasons which will become clear later, it is convenient to include more information by considering not only the conditional law of $\tau$ but rather the joint conditional law of $(W,\tau)$. Hence, we work with the canonical space
\begin{align}
\Omega := \C^2 \times \P(\C \times [0,T]) \times [0,T], \label{def:Omega}
\end{align}
and let $(B,W,\mu,\tau)$ denote the canonical process given by the natural projections, $(B,W) : \Omega \rightarrow \C^2$, $\mu : \Omega \rightarrow \P(\C \times [0,T])$, and $\tau : \Omega \rightarrow [0,T]$. Because we will work with a number of canonical filtrations on this space, we introduce the following notation which we shall use systematically in the sequel. The continuous processes $B$ and $W$ each generate filtrations $\FF^B=(\F^B)_{t \in [0,T]}$ and $\FF^W$, respectively, defined in the natural way. The random time $\tau$ generates the raw filtration
\[
\F^\tau_t = \sigma\{\tau \wedge t\}.
\]
The filtration generated by multiple processes is denoted, for instance, by $\FF^{B,W} := \FF^B \vee \FF^W$, or $\F^{B,W}_t = \sigma(\F^B_t \cup \F^W_t)$. 
We use the same notation $\FF^{W,\tau}$ not only for the filtration $\FF^W \vee \FF^\tau$ defined on $\Omega$, but also for the filtration generated on $\C \times [0,T]$, and this should not cause any confusion. With this identification, the filtration $\FF^\mu=(\F^\mu_t)_{t \in [0,T]}$ on $\Omega$ (or on $\P(\C \times [0,T])$) is defined by
\[
\F^\mu_t = \sigma\{\mu(C) : C \in \F^{W,\tau}_t\}.
\]
Equivalently, if $\pi_t$ is defined on $\C \times [0,T]$ by $\pi_t(w,s) = (w_{\cdot \wedge t}, s \wedge t)$, then $\F^\mu_t = \sigma\{\mu \circ \pi_t^{-1}\}$.
We also write
\[
\mu^W = \mu(\cdot \times [0,T]), \quad\quad \mu^\tau = \mu(\C \times \cdot),
\]
for the two marginals of $\mu$, which take values in $\P(\C)$ and $\P([0,T])$, respectively.
Given a filtration $\FF=(\F_t)_{t \in [0,T]}$, we write $\FF_+$ for the right-continuous filtration $(\F_{t+})_{t \in [0,T]}$, where as usual $\F_{t+} := \cap_{s > t}\F_s$ for $t \in [0,T)$ and $\F_{T+}=\F_T$.  Note that the right-filtration $\FF^\tau_+$ is the smallest filtration for which $\tau$ is a stopping time, and for this reason the appearance of right-continuous filtrations in the following definition quite natural:

\begin{definition} \label{def:weakMFE}
A \emph{weak mean field equilibrium (MFE)} is a probability measure $P$ on $\Omega$ such that:
\begin{enumerate}
\item $(B,W)$ is a Wiener process with respect to the full filtration $\FF^{B,W,\mu,\tau}_+$.
\item $(B,\mu)$ is independent of $W$.
\item $\tau$ is compatible with $(B,W,\mu)$, in the sense that $\F^\tau_{t+}$ is conditionally independent of $\F^{B,W,\mu}_T$ given $\F^{B,W,\mu}_{t+}$, for every $t \in [0,T]$.
\item The optimality condition holds:
\[
\E^{P}[F(B,W,\mu^\tau,\tau)] = \sup_{P'}\E^{P'}[F(B,W,\mu^\tau,\tau)],
\]
where the supremum is over all $P' \in \P(\Omega)$ satisfying (1-3) as well as $P' \circ (B,W,\mu)^{-1} = P \circ (B,W,\mu)^{-1}$.
\item The weak fixed point condition holds: $\mu = P\left((W,\tau) \in \cdot \ | \ B,\mu \right)$.
\end{enumerate}
\end{definition}

The following result is the first justification for the above definition, and after stating it we will elaborate further on the intuitive meaning of a weak MFE. Recall that $\W$ denotes Wiener measure, and write $\W^2 = \W \times \W$ for the product measure on $\C^2$.

\begin{proposition} \label{pr:strongisweak}
Assume that $F$ is bounded and jointly measurable and that $t \mapsto F(b,w,m,t)$ is continuous, for every $m$ and $\W^2$-almost every $(b,w)$. Suppose $\tau^*$ is a strong MFE, and define $\mu = \W^2(\tau^* \in \cdot | B)$. Then the measure
\begin{align}
P = \W^2 \circ (B,W,\mu,\tau^*)^{-1} \label{def:strongisweakP}
\end{align}
is a weak MFE.
\end{proposition}

The proof of Proposition \ref{pr:strongisweak} is in Section \ref{se:compatibility}. With some abuse of terminology, we may refer to the measure $P$ itself, defined in \eqref{def:strongisweakP}, as a strong MFE.
We may define also some intermediate notions of MFE. It may happen that $\tau$ is a.s. $(B,W,\mu)$-measurable under $P$, in which case we say $P$ is a \emph{weak MFE with strong stopping time}.\footnote{To say that $\tau$ is a.s. $(B,W,\mu)$-measurable under $P$ means that $\tau$ is measurable with respect to the $P$-completion of $\sigma(B,W,\mu)$. Equivalently, there exists a measurable map $\widetilde{\tau} : \C^2 \times \P(\C \times [0,T]) \rightarrow [0,T]$ such that $P(\tau=\widetilde{\tau}(B,W,\mu))=1$.} In contrast, we may refer to a weak MFE more verbosely as a \emph{weak MFE with weak stopping time}, to emphasize the failure of $\tau$ to be $(B,W,\mu)$-measurable. Likewise, we say that a weak MFE $P$ is a \emph{strong MFE with weak stopping time} if $\mu$ is $P$-a.s. $B$-measurable. A \emph{strong MFE with strong stopping time}, naturally, requires both of these measurability conditions, and according to Proposition \ref{pr:strongisweak} this reduces to what we have already called a \emph{strong MFE}.

The ``compatibility'' condition (3) of Definition \ref{def:weakMFE} is somewhat unusual. As mentioned above, we cannot expect $\tau$ to be $(B,W,\mu)$-measurable after taking weak limits, but conditions (3) captures an important structure we do retain, as does the requirement in (1) that $(B,W)$ remain Wiener processes with respect to the larger filtration.
Similar compatibility conditions were identified in the stochastic differential mean field games in \cite{lacker-mfglimit,carmonadelaruelacker-mfgcommonnoise} (see also \cite{CarmonaDelarue}), and indeed these notions of compatibility all fall under the same umbrella, which we clarify somewhat in Section \ref{se:compatibility}. Intuitively, our representative agent is allowed to randomize her stopping time externally to the signal $(B,W,\mu)$, as long as at each time $t$ this randomization is conditionally independent of all future information given the history of the signal. Mathematically, the reason compatibility arises is the following, stated informally here and made precise in Theorem \ref{th:nonradomizeddensity2}: given $\tau$ satisfying (3), there exists a sequence of $\FF^{B,W,\mu}$-stopping times $\tau_k$ such that $(B,W,\mu,\tau_k) \Rightarrow (B,W,\mu,\tau)$, where $\Rightarrow$ denotes convergence in law. In this sense, the set of compatible stopping times is the closure of the set of \emph{bona fide} stopping times.

\subsubsection*{\textbf{Continuous Objective Functions}}
We are nearly ready to state the main results of this section, but first we need some assumptions:

\begin{assumptionD*}
The function $F$ is bounded and jointly measurable, and $\P([0,T]) \times [0,T] \ni (m,t) \mapsto F(b,w,m,t)$ is continuous for $\W^2$-almost every $(b,w) \in \C^2$, when $\P([0,T])$ is equipped with the topology of weak convergence.
\end{assumptionD*}

The boundedness assumption is for convenience only, and this could easily be relaxed at the cost of some careful growth or integrability assumptions. The continuity assumption is important for our weak convergence methods, but unfortunately it can be restrictive. For instance, our bank run model in the introduction involved the discontinuous function $\P([0,T]) \times [0,T] \ni (m,t) \mapsto m[0,t]$. A close approximation of the bank run model could be accounted for nonetheless by replacing $m[0,t]$ by $\phi * m(t) = \int_{[0,T]}\phi(t-s)m(ds)$, where $\phi : [-T,T]\rightarrow \R$ is continuous and in some sense ``close to'' the step function $1_{[0,T]}$.

The first result is a limit theorem, stating that $n$-player equilibria converge to weak MFE. Recall the notation of the $n$-player game in Section \ref{se:nplayergame}. For each $n$ and each $t^1,\ldots,t^n \in [0,T]$ we define the random joint empirical measure (a measure on $\C \times [0,T]$)
\begin{align}
\widehat{\mu}^n(t^1,\ldots,t^n) = \frac{1}{n}\sum_{i=1}^n\delta_{(W^i,t^i)} .\label{def:empirical-measure-full}
\end{align}

\begin{theorem} \label{th:mainlimit}
Suppose assumption \textbf{D} holds.
Let $\epsilon_n \ge 0$ with $\epsilon_n \rightarrow 0$, and suppose $\vec{\tau}^{\,n}=(\tau^n_1,\ldots,\tau^n_n)$ is an $\epsilon_n$-Nash equilibrium for the $n$-player game for each $n$. 
Define 
\[
P_n = \frac{1}{n}\sum_{i=1}^n\PP \circ \left(B,W^i,\widehat{\mu}^n(\vec{\tau}^{\,n}),\tau^n_i\right)^{-1}.
\]
Then $(P_n)_{n=1}^\infty$ is tight, and every weak limit is a weak MFE.
\end{theorem}

The measure $P_n$ appearing in Theorem \ref{th:mainlimit} is quite a natural object to study, if interpreted the right way. We may write $P_n = \PP \circ (B,W^U,\widehat{\mu}^n(\vec{\tau}^{\,n}),\tau^n_U)^{-1}$, where $U$ is a random variable drawn uniformly from $\{1,\ldots,n\}$, independent of $(B,W^i)_{i=1}^\infty$. Think of this as a \emph{randomly selected representative agent}. As $\tau^n_i$ may fail to be symmetric in any useful sense, one would not get far by working with, say, $\PP \circ (B,W^1,\widehat{\mu}^n(\vec{\tau}^{\,n}),\tau^n_1)^{-1}$, which corresponds to \emph{arbitrarily} choosing agent $1$ as the representative. The same idea appears in the following converse to Theorem \ref{th:mainlimit}, which is an analog of Theorem \ref{th:main-converse} for the case of weak equilibria.

\begin{theorem} \label{th:converselimit}
Suppose assumption \textbf{D} holds. Let $P$ be a weak MFE. Then there exist  $\epsilon_n \rightarrow 0$ and $\epsilon_n$-Nash equilibria $\vec{\tau}^{\,n}=(\tau^n_1,\ldots,\tau^n_n)$ such that
\[
P = \lim_{n\rightarrow\infty}\frac{1}{n}\sum_{i=1}^n\PP \circ \left(B,W^i,\widehat{\mu}^n(\vec{\tau}^{\,n}),\tau^n_i\right)^{-1}.
\]
In fact, if $\tau^*=\tau^*(B,W)$ is a strong MFE in the sense of Definition \ref{def:strongMFE}, then we can take $\vec{\tau}^{\,n}$ of the form $\tau^n_i = \tau^*(B,W^i)$.
\end{theorem}

Finally, we state an existence result for weak MFE. Combined with Theorem \ref{th:converselimit}, it shows that approximate $n$-player equilibria exist for the $n$-player games.

\begin{theorem} \label{th:existence}
Under assumption \textbf{D}, there exists a weak MFE.
\end{theorem}

Some comments are in order at this stage. Combining the two limit theorems tells us that the set of weak MFEs is precisely the set of limits of $n$-player approximate equilibria. If we can find a strong MFE $\tau^*$, the converse limit theorem \ref{th:converselimit} shows how to construct from it an approximate $n$-player equilibria in a pleasantly symmetric and distributed form, as in Theorem \ref{th:main-converse}. The general structure of the results and arguments are  similar to \cite{lacker-mfglimit,carmonadelaruelacker-mfgcommonnoise}.

\section{Proofs in the general setup} \label{se:proofs-supermodular}

This section proves the results of Section \ref{se:results}.
Throughout this section, we work on the space $(\overline{\Omega},\overline{\F},\overline{\FF},\overline{\PP})$ defined in Section \ref{se:nplayergame}. With some abuse of notation, any function $\phi$ on $\Omega^{\mathrm{com}} \times \Omega^{\mathrm{ind}}$ is automatically extended to all of $\overline{\Omega}$ by setting
\[
\phi(\omega^0,\omega^1,\ldots) := \phi(\omega^0,\omega^1).
\]

\subsection*{Proof of Theorem \ref{th:main-converse}} \label{se:proof-converse}

Abbreviate $\vec{\tau}^{\,n,-k} := (\vec{\tau}^{\,n})^{-k}$, and define
\begin{align*}
\epsilon_n = \sup_{\widetilde{\tau}}\E\left[F(W^0,W^1,\overline{\mu}^n(\vec{\tau}^{\,n,-1},\widetilde{\tau}) ,\widetilde{\tau})\right] - \E\left[F(W^0,W^1,\overline{\mu}^n(\vec{\tau}^{\,n}) ,\tau^1)\right],
\end{align*}
where the supremum is over $\FF^1$-stopping times. Clearly $\epsilon_n \ge 0$. By symmetry, the index $1$ could be replaced by any $k\in\{1,\ldots,n\}$. Hence, $\vec{\tau}^{\,n}$ is an $\epsilon_n$-Nash equilibrium for each $n$. We must only show that $\epsilon_n \rightarrow 0$.

First we show that
\begin{align}
\lim_{n\rightarrow\infty}\E\left[\sup_{t \in \T}\left|F(W^0,W^1,\overline{\mu}^n(\vec{\tau}^{\,n}),t) - F(W^0,W^1,\mu,t)\right|\right] = 0. \label{pf:limit1}
\end{align}
Note that the basic open sets of $\sigma(\P(\T),B(\T))$ are of the form
\[
U = \left\{m \in \P(\T) : \left|\int_{\T}\varphi_i\,d(m-\tilde{m})\right| < \epsilon, \ i=1,\ldots,k\right\},
\]
for $k \ge 1$, $\epsilon > 0$, and $\varphi_1,\ldots,\varphi_k \in B(\T)$.
Because $(\tau^k=\tau^*(W^0,W^k))_{k=1}^\infty$ are conditionally i.i.d. given $W^0$, and their common conditional law is $\mu=\mu(W^0)$, the law of large numbers yields
\begin{align}
\PP\left(\left.\overline{\mu}^n(\vec{\tau}^{\,n})  \notin U \right| W^0=\omega^0\right) \rightarrow 0, \label{pf:converse1}
\end{align}
for almost every $\omega^0$, for every basic $\sigma(\P(\T),B(\T))$-open neighborhood $U$ of $\mu(\omega^0)$.
Thanks to the continuity assumption (A.2), for each $\delta > 0$ and almost every $(\omega^0,\omega^1)$ we can find a basic $\sigma(\P(\T),B(\T))$-open neighborhood $U$ of $\mu(\omega^0)$ such that $\nu \in U$ implies
\[
\sup_{t \in \T}\left|F(\omega^0,\omega^1,\nu,t) - F(\omega^0,\omega^1,\mu(\omega^0),t)\right| < \delta.
\]
Thus, for a.e. $\omega^0$,
\[
\PP\left(\left.\sup_{t \in \T}\left|F(W^0,W^1,\overline{\mu}^n(\vec{\tau}^{\,n}),t) - F(W^0,W^1,\mu(W^0),t)\right| \ge \delta \right| W^0=\omega^0\right) \rightarrow 0.
\]
Thanks to assumption (A.3), the limit \eqref{pf:limit1} follows from dominated convergence.

Next, we argue that
\begin{align}
\lim_{n\rightarrow\infty}\sup_{\widetilde{\tau}}\E\left[F(W^0,W^1,\overline{\mu}^n(\vec{\tau}^{\,n,-1},\widetilde{\tau}) ,\widetilde{\tau})\right] = \sup_{\widetilde{\tau}}\E\left[F(W^0,W^1,\mu,\widetilde{\tau})\right]. \label{pf:limit2}
\end{align}
Indeed, using the easy estimate
\[
\sup_{t \in \T}\|\overline{\mu}^n(\vec{\tau}^{\,n,-1},t) - \overline{\mu}^n(\vec{\tau}^{\,n})\|_{TV} \le 2/n,
\]
along with \eqref{pf:converse1}, we deduce that for almost every $\omega^0$ and for every basic $\sigma(\P(\T),B(\T))$-open neighborhood $U$ of $\mu(\omega^0)$, we have
\[
\lim_{n\rightarrow\infty}\PP\left(\left.\overline{\mu}^n(\vec{\tau}^{\,n,-1},t)  \notin U, \text{ for some } t\in\T \right| W^0=\omega^0\right) = 0.
\]
Repeat the argument leading to \eqref{pf:limit1} above to get
\begin{align*}
\E\left[\sup_{t \in \T}\left|F(W^0,W^1,\overline{\mu}^n(\vec{\tau}^{\,n,-1},t) ,t) - F(W^0,W^1,\mu,t)\right|\right] \rightarrow 0.
\end{align*}
Finally, we conclude from \eqref{pf:limit1}, \eqref{pf:limit2}, and the optimality of $\tau^*$ that 
\begin{align*}
\lim_{n\rightarrow\infty}\epsilon_n &= \sup_{\widetilde{\tau}}\E\left[F(W^0,W^1,\mu,\widetilde{\tau})\right] - \E\left[F(W^0,W^1,\mu,\tau^*)\right] = 0.
\end{align*}
\hfill \qedsymbol

\subsection*{Proof of Lemma \ref{le:glivenko-cantelli}}
Fix $\epsilon > 0$. Find a finite set $0 = t_0 < t_1 < \cdots < t_N = \sup\T$ such that $m(t_k,t_{k+1}) \le \epsilon$ for every $k=0,\ldots,N-1$. Consider the $\sigma(\P(\T),B(\T))$-open neighborhood $U$ of $m_0$ given by
\begin{align*}
U = \left\{m \in \P(\T) : |m(t_k,t_{k+1})  - m_0(t_k,t_{k+1})| \vee |m[0,t_k] - m_0[0,t_k]|  < \epsilon, \ \forall k=0,\ldots,N-1\right\}.
\end{align*}
For $m \in U$ and $t \in (t_k,t_{k+1})$ we have
\begin{align*}
|m[0,t] - m_0[0,t]| &\le |m[0,t] - m[0,t_k]|+ |m[0,t_k] - m_0[0,t_k]| + |m_0[0,t_k] - m_0[0,t]| \\
	&\le m(t_k,t] + m_0(t_k,t]+ |m[0,t_k] - m_0[0,t_k]| \\
	&\le m(t_k,t_{k+1}) + m_0(t_k,t_{k+1}) + |m[0,t_k] - m_0[0,t_k]| \\
	&\le 3\epsilon.
\end{align*}
Setting $\pi = \{t_0,\ldots,t_{N-1}\}$, we have, for $m \in U$,
\begin{align*}
\sup_{t \in \T}|m[0,t] - m_0[0,t]| &= \sup_{t \notin \pi}|m[0,t] - m_0[0,t]| \vee \max_{t \in \pi}|m[0,t] - m_0[0,t]| \le 3\epsilon.
\end{align*}
\hfill\qedsymbol

\subsection*{Existence under supermodularity}

In this section, we prove Theorem \ref{th:existence-B}.
Let $\SS$ denote the set of (equivalence classes of a.s. equal) $\FF^{\mathrm{sig}}$-stopping times, and let $\M$ denote the set of (equivalence classes of a.s. equal) $\P(\T)$-valued random variables $\mu$, which are $\FF^{\mathrm{com}}$-adapted in the sense that $\mu[0,t]$ is a.s. $\F^{\mathrm{com}}_t$-measurable for each $t$. 
Equip $\SS$ with the almost sure partial order, meaning that we interpret the inequality $\tau \le \tau'$ as holding almost surely.
Equip $\M$ with the almost sure stochastic order, meaning that $\mu' \ge \mu$ if and only if $\mu'[0,t] \le \mu[0,t]$ a.s. for each $t \in [0,T]$, and note that right-continuity renders the order of quantifiers inconsequential.
Note that  $\M$ is a  lattice, namely a partially ordered set in which every two elements have a unique least upper bound and a unique greatest lower bound; for example $\mu \vee \mu'$ is the random measure defined by $(\mu \vee \mu')[0,t] = \mu[0,t] \wedge \mu'[0,t]$. On the other hand, $\SS$ is a \emph{complete} lattice in the sense that it is a partially ordered set in which every subset has both a supremum and an infimum. Indeed, the notion of ``essential supremum'' provides the correct supremum operation on $\SS$, and the completeness follows from the assumption (B.1) that the filtration $\FF^{\mathrm{sig}}$ is right-continuous. The completeness of the lattice of stopping times is surely known, but we prove it in the Appendix (Theorem \ref{th:stoppingtime-completelattice}) as we were unable to locate a precise reference.

Now define $J : \M \times \SS \rightarrow \R$ by
\[
J(\mu,\tau) = \E[F(B,W,\mu,\tau)].
\]
Note that $J(\mu,\tau)$ is trivially supermodular in $\tau$, in the sense that 
\[
J(\mu,\tau \vee \tau') + J(\mu,\tau \wedge \tau') \ge J(\mu,\tau) + J(\mu,\tau'),
\]
for every $\mu \in \M$ and every pair $\tau,\tau' \in \SS$. In fact, this holds with equality, which follows from taking expectations on both sides of the identity
\begin{align*}
F(B,W,\mu,\tau \vee \tau') + F(B,W,\mu,\tau \wedge \tau') = F(B,W,\mu,\tau) + F(B,W,\mu,\tau').
\end{align*}
Moreover, assumption (B.2) ensures that $J$ has increasing differences with respect to $\mu$, in the sense that
\[
J(\mu',\tau') - J(\mu',\tau) \ge J(\mu,\tau') - J(\mu,\tau),
\]
whenever $\tau,\tau' \in \SS$ and $\mu,\mu' \in \M$ satisfy $\tau \le \tau'$ and $\mu \le \mu'$. 
From Topkis's monotonicity theorem \cite{milgrom1990rationalizability}, we deduce that the set-valued map
\[
\Phi(\mu) :=  \arg\max_{\tau \in \SS}J(\mu,\tau)
\]
is increasing in the strong set order, meaning that whenever $\mu,\mu' \in \M$ satisfy $\mu \le \mu'$, and whenever $\tau \in \Phi(\mu)$ and $\tau' \in \Phi(\mu')$, we have $\tau \vee \tau' \in\Phi(\mu')$ and $\tau \wedge \tau' \in \Phi(\mu)$.
It is readily checked that $J$ is order upper semicontinuous in $\tau$, using the assumed upper semicontinuity of $F$ in $\tau$ along with Fatou's lemma, justified by the integrability assumption \eqref{def:A.5}.
By \cite[Theorem 1]{milgrom1990rationalizability}, this implies that for every $\mu$, $\Phi(\mu)$ is a nonempty complete lattice. In particular, it has a maximum, which we denote $\phi^*(\mu)$ and a minimum which we denote by $\phi_*(\mu)$. Note that $\phi^* : \M \rightarrow \SS$ is increasing in the sense that $\mu' \ge \mu$ implies $\phi^*(\mu') \ge \phi^*(\mu)$. 
Moreover, it is plain to check that the function $\psi : \SS \rightarrow \M$ defined by $\psi(\tau) = \text{Law}(\tau | W^0)$ is monotone. Thus $\phi^* \circ \psi$ is a monotone map from $\SS$ to itself, and since $\SS$ is a complete lattice we conclude from Tarski's fixed point theorem that there exists $\tau$ such that $\tau = \phi^*(\psi(\tau))$. It is readily verified that any such fixed point $\tau$ is a strong MFE, in the sense of Definition \ref{def:strongMFE}.

Now, under the additional assumption \textbf{C}, we prove the last claim in the statement of Theorem \ref{th:existence-B}.
Define $\tau_0 \equiv \sup\T$, and by induction $\tau_n=\phi^*\circ\psi(\tau_{n-1})$ for $n \ge 1$. 
Clearly, $\tau_1\le \tau_0$. Now if we assume $\tau_n\le \tau_{n-1}$, then the monotonicity of $\phi^*\circ\psi$ proved earlier implies $\tau_{n+1}=\phi^*\circ\psi(\tau_n)\le\phi^*\circ\psi(\tau_{n-1})=\tau_n$. If we define $\tau^*$ as the a.s. limit of the nonincreasing sequence $(\tau_n)_{n\ge 1}$ of stopping times, then $\tau^*\in \SS$ because the lattice $\SS$ is complete (see Theorem \ref{th:stoppingtime-completelattice}). We claim that $\tau^*$ is a MFE. Note that $\psi(\tau_n) \rightarrow \psi(\tau^*)$ weakly almost surely, because $\tau_n \rightarrow \tau^*$. The assumption (C.1) of joint continuity of $F=F(\omega^0,\omega^1,m,t)$ in $(m,t)$, along with the uniform integrability assumption (C.2) ensure by dominated convergence that
\begin{align*}
J(\psi(\tau^*),\tau^*) = \lim_{n\rightarrow\infty}J(\psi(\tau_n),\tau_{n+1}).
\end{align*}
Moreover, for any $\sigma \in \SS$, the fact that $\tau_{n+1} \in \Phi(\psi(\tau_n))$ implies
\begin{align*}
J(\psi(\tau_n),\tau_{n+1}) \ge J(\psi(\tau_n),\sigma).
\end{align*}
Pass to the limit on both sides to get
\begin{align*}
J(\psi(\tau^*),\tau^*) \ge J(\psi(\tau^*),\sigma).
\end{align*}
This shows $\tau^* \in \Phi(\psi(\tau^*))$, and in particular $\tau^*$ is a MFE.

Similarly, define $\theta_0 \equiv 0$, and by induction $\theta_n=\phi_*\circ\psi(\theta_{n-1})$ for $n \ge 1$. 
Clearly, $\theta_1\ge \theta_0$, and as above, we prove by induction that $\theta_n\ge \theta_{n-1}$.
Next, we define $\theta^*$ as the a.s. limit of the nondecreasing sequence $(\theta_n)_{n\ge 1}$ of stopping times. Conclude as before that $\theta^*$ is a MFE.

Finally, it is plain to check that if $\tau$ is any MFE, it is a fixed point of the set-valued map $\Phi \circ \psi$, in the sense that $\tau \in \Phi(\psi(\tau))$. Trivially, $\theta_0=0\le\tau\le \sup\T = \tau_0$. Applying $\phi_*\circ\psi$ and $\phi^*\circ\psi$ repeatedly to the left and right sides, respectively, we conclude that $\theta_n \le \tau \le \tau_n$ for each $n$, and thus $\theta^* \le \tau \le \tau^*$.
\hfill\qedsymbol

\begin{remark}
The above proof shows that, under the full continuity assumption, there is no need to use Tarski's theorem to prove existence, as the MFEs $\tau^*$ and $\theta^*$ are constructed inductively.
\end{remark}


\section{Compatibility and the density of non-randomized stopping times} \label{se:compatibility}
This section elaborates on the crucial notion of \emph{compatibility} introduced in property (3) of Definition \ref{def:weakMFE} and, in doing so, takes some first steps toward proving the results announced in Section \ref{se:weakEQ}.
Here, our goal is to discuss some important facts about these compatibility properties, namely how to approximate compatible (randomized) stopping times with nonrandomized stopping times.
Essentially, this has to do with filtration enlargements. To say that $\F^\tau_{t+}$ is conditionally independent of $\F^{B,W,\mu}_T$ given $\F^{B,W,\mu}_{t+}$ is the same as saying that $\F^{B,W,\mu,\tau}_{t+}$ is conditionally independent of $\F^{B,W,\mu}_T$ given $\F^{B,W,\mu}_{t+}$. To say that this holds for every $t \in [0,T)$, it turns out, is equivalent to saying that every $\FF^{B,W,\mu}_+$-martingale remains a $\FF^{B,W,\mu,\tau}_+$-martingale.
Many different names and characterizations are associated to this property of a filtration enlargement, such as the H-hypothesis \cite{bremaudyor-changesoffiltrations}, immersion \cite{jeanblanc2009immersion}, very good extensions \cite{jacodmemin-weaksolution}, and natural extensions \cite{elkaroui-compactification}, while we borrow the term \emph{compatible} from Kurtz \cite{kurtz-yw2013}, to be consistent with other works on mean field games \cite{lacker-mfglimit,carmonadelaruelacker-mfgcommonnoise, CarmonaDelarue}.
Before we proceed, we recall a useful result on weak convergence which will be used repeatedly:

\begin{lemma}[Corollary 2.9 and Theorem 2.16 of \cite{jacodmemin-stable}] \label{le:jacodmemin}
Suppose $E$ and $E'$ are Polish spaces. Suppose $P_n,P \in \P(E \times E')$ satisfy $P_n \rightarrow P$, and suppose that every $P_n$ has the same $E$-marginal. That is, $P_n(\cdot \times E')$ does not depend on $n$. Then, for every bounded measurable function $\phi : E \times E' \rightarrow \R$ such that $\phi(x,\cdot)$ is continuous on $E'$ for $\mu$-almost every $x \in E$, we have
\[
\int \phi\,dP_n \rightarrow \int \phi\,dP.
\]
\end{lemma}

Of utmost importance to us is the behavior of compatibility under weak limits of the underlying probability measures. The key result is the following, which says roughly that a compatible process is the weak limit of adapted processes. 

\begin{proposition} \label{pr:density-discrete}
Let $\X$ and $\Y$ be Polish spaces, with $\X$ homeomorphic to a convex subset of a locally convex space. Let $Y=(Y_1,\ldots,Y_N)$ and $X=(X_1,\ldots,X_N)$ be $\Y$- and $\X$-valued stochastic processes, respectively, defined on a common probability space. For $R \in \{X,Y\}$, let $\F^R_n = \sigma\{R_1,\ldots,R_n\}$ denote the filtration generated by $R$. Assume that the law of $Y_1$ is nonatomic. Lastly, assume that $X$ is compatible with $Y$ in the sense that $\F^X_n$ is conditionally independent of $\F^Y_N$ given $\F^Y_n$, for each $n=1,\ldots,N$. 
Then there exist continuous functions $h^j_k : \Y^k \rightarrow \X$, for $k \in \{1,\ldots,N\}$ and $j \ge 1$, such that
\[
(Y,(h^j_1(Y_1),h^j_2(Y_1,Y_2),\ldots,h^j_N(Y_1,\ldots,Y_N))) \rightarrow (Y,X)
\]
in law in the space $\Y^N \times \X^N$, as $j\rightarrow\infty$.
In particular, there exist $Y$-adapted $\X$-valued processes $X^j=(X^j_1,\ldots,X^j_N)$ such that $(Y,X^j) \Rightarrow (Y,X)$.
\end{proposition}
\begin{proof}
See Appendix \ref{se:adapted-density-proof}.
\end{proof}

\subsection{Randomized stopping times} \label{se:randomizedstoppingtimes}

This section is devoted to some compactness properties of randomized stopping times, analogous to, but extending results of Baxter and Chacon \cite{baxter-chacon}.
Abbreviate
\begin{align*}
\Omega_{\mathrm{input}} = \C^2 \times \P(\C \times [0,T]).
\end{align*}
For this section, fix a measure $\rho \in \P(\Omega_{\mathrm{input}})$, to represent a joint law of $(B,W,\mu)$, and assume, under $\rho$, that $(B,W)$ are Wiener processes with respect to the filtration $\FF^{B,W,\mu}$ (and thus also with respect to the right-filtration $\FF^{B,W,\mu}_+$). Note that $\Omega = \Omega_{\mathrm{input}} \times [0,T]$.

We next define three sets of probability measures on $\Omega$, corresponding to various notions of (randomized) stopping times:
\begin{itemize}
\item $\RC^+(\rho)$ is the set of joint laws $P \in \P(\Omega)$ with $\Omega_{\mathrm{input}}$-marginal $\rho$ such that $\F^\tau_{t+}$ is conditionally independent of $(B,W,\mu)$ given $\F^{B,W,\mu}_{t+}$ for every $t \in [0,T)$. That is, $\RC^+(\rho)$ is the set of $P$ satisfying the compatibility property (3) of Definition \ref{def:weakMFE}.
\item $\RC(\rho)$ is the set of joint laws $P \in \P(\Omega)$ with $\Omega_{\mathrm{input}}$-marginal $\rho$ such that $\F^\tau_t$ is conditionally independent of $(B,W,\mu)$ given $\F^{B,W,\mu}_t$, for every $t \in [0,T)$.
\item $\RC_0(\rho)$ is the set of $P \in \P(\Omega)$ with $\Omega_{\mathrm{input}}$-marginal $\rho$ under which $\tau$ is a stopping time relative to the $P$-completion of $\FF^{B,W,\mu}$, and moreover $\tau$ is of the form $\tau = \hat{\tau}(B,W,\mu)$ for some \emph{continuous} function $\hat{\tau} : \Omega_{\mathrm{input}} \rightarrow [0,T]$.
\end{itemize}
Both sets $\RC^+(\rho)$ and $\RC(\rho)$ represent slightly different notions of randomized stopping time, though we will soon see that $\RC^+(\rho)=\RC(\rho)$.
On the other hand, $\RC_0(\rho)$ should be seen as the set of (joint laws of) bona fide $\FF^{B,W,\mu}$-stopping times, with the useful additional property that $\tau$ can be written as a continuous function of $(B,W,\mu)$.

\begin{remark} \label{re:wiener-filtration}
Suppose $\rho(db,dw,dm) = \W(db)\W(dw)\delta_{\widehat{m}(b)}(dm)$ for some measurable function $\widehat{m} : \C \rightarrow \P(\C \times [0,T])$. Suppose $\widehat{m}$ is \emph{adapted} in the sense that $b \mapsto \widehat{m}(b)(C)$ is $\F^B_t$-measurable whenever $C \in \F^{W,\tau}_t$, for $t \in [0,T]$. Then, under $\rho$, $\F^{B,W,\mu}_t=\F^{B,W}_t$ a.s., for every $t$. It is then easy to argue that $\RC^+(\rho)$ (resp. $\RC(\rho)$) is precisely the set of joint laws $P \in \P(\Omega)$ with $\Omega_{\mathrm{input}}$-marginal $\rho$ such that $(B,W)$ is a Wiener process with respect to the full filtration $\FF^{B,W,\mu,\tau}_+$ (resp. $\FF^{B,W,\mu,\tau}$). Indeed, because $\F^{B,W}_T$ can be split into two independent parts, $\F^{B,W}_T = \sigma\{(B_s-B_t,W_s-W_t) : s \in [t,T]\} \vee \F^{B,W}_t$, it holds that $\F^\tau_t$ is conditionally independent of $\F^{B,W}_T$ given $\F^{B,W}_t$ if and only if $\F^{B,W,\tau}_t$ is independent of $\sigma\{(B_s-B_t,W_s-W_t) : s \in [t,T]\}$.
\end{remark}

\begin{theorem} \label{th:nonradomizeddensity2}
For $\rho$ as above, $\RC^+(\rho)$ is convex and compact and equals the closure of $\RC_0(\rho)$. Moreover,  $\RC^+(\rho) = \RC(\rho)$.
\end{theorem}

Before turning to the proof, we state an immensely useful corollary:

\begin{corollary} \label{co:same-supremum}
Assume that $F$ is bounded and jointly measurable and that $t \mapsto F(b,w,m,t)$ is continuous, for every $m$ and $\W^2$-almost every $(b,w)$.
For $\rho$ as above, we have
\[
\sup_{P \in \RC^+(\rho)}\E^P[F(B,W,\mu^\tau,\tau)] = \sup_{P \in \RC_0(\rho)}\E^P[F(B,W,\mu^\tau,\tau)].
\]
\end{corollary}
\begin{proof}
By Theorem \ref{th:nonradomizeddensity2}, $\RC_0(\rho)$ is dense in $\RC^+(\rho)$. It suffices to show that $P \mapsto \E^P[F(B,W,\mu,\tau)]$ is continuous on $\RC^+(\rho)$. But this follows from the assumption on $F$ and from Lemma \ref{le:jacodmemin}.
\end{proof}

We precede the proof of Theorem \ref{th:nonradomizeddensity2} with a preparatory lemma, which allows us to map continuously between stopping times and c\`adl\`ag processes of a certain form. In the rest of this section, let $D = D([0,T];\R_+)$ denote the set of c\`adl\`ag functions (i.e., right-continuous with left limits) functions from $[0,T]$ to $\R_+=[0,\infty)$. Endow $D$ with the usual Skorohod $J_1$ topology. As usual, for $h \in D$, write $h(t-) = \lim_{s \uparrow t}h(s)$ for $t \in (0,T]$ and $h(0-)=h(0)$.

\begin{lemma} \label{le:Phi-continuous}
Define $\Phi : D \rightarrow [0,T]$ by 
\[
\Phi(h) = \inf\left\{t \ge 0 : h(t) \ge 1/2\right\} \wedge T.
\]
Let $\SS$ denote the set of nondecreasing $h \in D$ for which $h(t-)\le 1/2 \le h(t)$ implies $t = \Phi(h)$. Then $\Phi$ is continuous at each point of $\SS$.
\end{lemma}
\begin{proof}
Let $h_n \rightarrow h$ in $D$, where $h \in \SS$. Let $s_n = \Phi(h_n)$, and note that $(s_n)_{n=1}^\infty$ is bounded. Let $(s_{n_k})_{k=1}^\infty$ denote any convergent subsequence, and let $s \in [0,T]$ denote its limit. Assume first that $0 < s < T$, so that without loss of generality we may take $0 < s_{n_k} < T$ for every $k$.
Because $h_{n_k} \rightarrow h$ and $s_{n_k} \rightarrow s$, it follows that $(h_{n_k}(s_{n_k}))_{k=1}^\infty$ is bounded, and its limit points are contained in $\{h(s-),h(s)\}$ (see \cite[Proposition 3.6.5]{ethier-kurtz}). Because $h_{n_k}(s_{n_k}) \ge 1/2$ for every $k$ and $h(s) \ge h(s-)$ (as $h \in \SS$), we conclude that $h(s) \ge 1/2$. On the other hand, for $\epsilon > 0$, $(h_{n_k}(s_{n_k}-\epsilon))_{k=1}^\infty$ is bounded, and its limit points are contained in $\{h((s-\epsilon)-),h(s-\epsilon)\}$. Because $h_{n_k}(s_{n_k}-\epsilon) < 1/2$ for every $k$ and $h((s-\epsilon)-) \le h(s-\epsilon)$, we conclude that $h(s-\epsilon) \le 1/2$. Sending $\epsilon \downarrow 0$ yields $h(s-) \le 1/2 \le h(s)$ and thus $s = \Phi(h)$.

Next, suppose $s = T$. Then again $(h_{n_k}(s_{n_k}-\epsilon))_{k=1}^\infty$ is bounded, and its limit points are contained in $\{h((T-\epsilon)-),h(T-\epsilon)\}$. Because $h_{n_k}(s_{n_k}-\epsilon) < 1/2$ for each $k$ and $h((T-\epsilon)-)\le h(T-\epsilon)$, we conclude that $h((T-\epsilon)-) \le 1/2$. This implies $h(T-) \le 1/2$, which is enough to show that $\Phi(h) = T$; indeed, either $h(T) \ge 1/2$, in which case $\Phi(h)=T$ because $h \in \SS$, or $h(T) < 1/2$, in which case $h(t) < 1/2$ for all $t \in [0,T]$ and again $\Phi(h) = T$.

Finally, suppose $s = 0$. Then $h_{n_k}(s_{n_k}) \rightarrow h(0)=h(0-)$, which shows $h(0) \ge 1/2$. Thus $\Phi(h)=0$.
\end{proof}

Before we begin the proof of Theorem \ref{th:nonradomizeddensity2}, notice that $\F^{B,W,\mu}_t = \sigma\{B_{\cdot \wedge t},W_{\cdot \wedge t},\mu^t\}$, where $m^t$ denotes the image of a measure $m\in \P(\C \times [0,T])$ through the map $\C \times [0,T] \ni (w,s) \mapsto (w_{\cdot \wedge t}, s \wedge t)$. This makes it clear that $\F^{B,W,\mu}_t$ is generated by the \emph{continuous} $\F^{B,W,\mu}_t$-measurable functions. Similarly, $\F^\tau_t = \sigma\{\tau \wedge t\}$ is generated by the \emph{continuous} $\F^\tau_t$-measurable functions.

\subsection*{Proof of Theorem \ref{th:nonradomizeddensity2}}
We break the proof up into four claims:

\textbf{$\bm{\RC(\rho)}$ is compact:}
Because the $\Omega_{\mathrm{input}}$-marginal of any element of $\RC(\rho)$ is $\rho$, and because $[0,T]$ is compact, it is immediate that $\RC(\rho)$ is tight.
To show $\RC(\rho)$ is closed, let $P_n \rightarrow P$ in $\P(\Omega_{\mathrm{input}})$, with $P_n \in \RC(\rho)$. Let $f_t$, $g_T$, and $g_t$ be bounded continuous functions on $[0,T]$, $\Omega_{\mathrm{input}}$, and $\Omega_{\mathrm{input}}$, respectively, and assume they are measurable with respect to $\F^\tau_t$, $\F^{B,W,\mu}_T$, and $\F^{B,W,\mu}_t$. Find a bounded $\F^{B,W,\mu}_t$-measurable function $\phi_t$ on $\Omega_{\mathrm{input}}$ such that $\phi_t(B,W,\mu) = \E^P[g_T(B,W,\mu) | \F^{B,W,\mu}_t]$. Because $P_n \circ (B,W,\mu)^{-1} = P \circ (B,W,\mu)^{-1} = \rho$ for every $n$, we have $\phi_t(B,W,\mu) = \E^{P_n}[g_T(B,W,\mu) | \F^{B,W,\mu}_t]$. Thus, by Lemma \ref{le:jacodmemin},
\begin{align*}
\E^P[f_t(\tau)g_T(B,W,\mu)g_t(B,W,\mu)] &= \lim_{n\rightarrow\infty}\E^{P_n}[f_t(\tau)g_T(B,W,\mu)g_t(B,W,\mu)] \\
	&= \lim_{n\rightarrow\infty}\E^{P_n}[f_t(\tau)\phi_t(B,W,\mu)g_t(B,W,\mu)] \\
	&= \E^{P}[f_t(\tau)\phi_t(B,W,\mu)g_t(B,W,\mu)].
\end{align*}
As remarked before the proof, the continuous bounded functions generate $\F^{B,W,\mu}_t$ and $\F^\tau_t$, and we concude that 
\begin{align*}
\E^P[f_t(\tau)g_T(B,W,\mu)g_t(B,W,\mu)] = \E^{P}[f_t(\tau)\E^P[g_T(B,W,\mu) | \F^{B,W,\mu}_t]g_t(B,W,\mu)],
\end{align*}
for all bounded functions $f_t$, $g_T$, $g_t$ with the same measurability requirements as above, but without the continuity requirements. This shows that $\F^\tau_t$ is conditionally independent of $\F^{B,W,\mu}_T$ given $\F^{B,W,\mu}_t$ under $P$, for every $t \in [0,T)$, so $P \in \RC(\rho)$.

\textbf{$\bm{\RC(\rho)=\RC^+(\rho)}$:}
First we show $\RC(\rho)\subset\RC^+(\rho)$. Fix $t \in [0,T)$ and $P \in \RC(\rho)$. For $A \in \F^{B,W,\mu}_T$ and $C \in \F^\tau_t$, we have
\[
\PP(C | \F^{B,W,\mu}_t)\PP(A | \F^{B,W,\mu}_t) =  \PP(C \cap A | \F^{B,W,\mu}_t).
\]
By backward martingale convergence, taking decreasing limits in $t$ yields
\[
\PP(C | \F^{B,W,\mu}_{t+})\PP(A | \F^{B,W,\mu}_{t+}) =  \PP(C \cap A | \F^{B,W,\mu}_{t+}).
\]
This shows that $\RC(\rho) \subset \RC^+(\rho)$, and we know from before that $\RC(\rho)$ is closed. Hence, it suffices to show that every point $P \in \RC^+(\rho)$ is the limit point of a sequence in $\RC(\rho)$. To see this, set
\[
P_n := P \circ (B,W,\mu,(\tau+1/n)\wedge T)^{-1},
\]
which converges weakly to $P$. Because $\F^\tau_{t+}$ is conditionally independent of $\F^{B,W,\mu}_T$ given  $\F^{B,W,\mu}_{t+}$ under $P$, we have, for $0 < s \le t < T$,
\begin{align*}
P_n(\tau \le s | \F^{B,W,\mu}_T) &= P(\tau \le s - 1/n | \F^{B,W,\mu}_T) = P(\tau \le s - 1/n | \F^{B,W,\mu}_{(t-1/n)+}) \\
	&= P_n(\tau \le s | \F^{B,W,\mu}_{(t-1/n)+}).
\end{align*}
Because $\F^{B,W,\mu}_{(t-1/n)+} \subset \F^{B,W,\mu}_t$, we conclude that
\begin{align}
P_n(\tau \le s | \F^{B,W,\mu}_T)  = P_n(\tau \le s | \F^{B,W,\mu}_t). \label{pf:density0.1}
\end{align}
Since $P_n(\tau=0)=0$, we also have \eqref{pf:density0.1} for $s=0$. Since \eqref{pf:density0.1} holds for all $s \in [0,t]$, we conclude that $\F^\tau_{t+} = \sigma\{\{\tau \le s\} : s \le t\}$ is independent of $\F^{B,W,\mu}_T$ given $\F^{B,W,\mu}_t$ under $P_n$, for $t \in (0,T)$. To conclude that $P_n \in \RC(\rho)$, we must still check that $\F^\tau_0$ is conditionally independent of $\F^{B,W,\mu}_T$ given  $\F^{B,W,\mu}_0$ under $P_n$. But this is obvious, as $\F^\tau_0= \sigma\{\tau \wedge 0\}$ is the trivial $\sigma$-field.

\textbf{$\bm{\RC(\rho)}$ is convex:}
To check that $\RC(\rho)$ is convex, note that $\RC(\rho)$ is the set of $P \in \P(\Omega_{\mathrm{input}} \times [0,T])$ with first marginal $\rho$ for which
\[
\E^P\left[\phi_t(\tau)\psi(B,W,\mu)\psi_t(B,W,\mu)\right] = \E^P\left[\E^P\left[\left.\phi_t(\tau)\right| \F^{B,W,\mu}_t\right]\psi(B,W,\mu)\psi_t(B,W,\mu)\right],
\]
for every $t \in [0,T]$ and every triple of bounded functions $\phi_t$, $\psi$, and $\psi_t$, measurable with respect to $\F^\tau_{t}$, $\F^{B,W,\mu}_T$, and $\F^{B,W,\mu}_{t}$, respectively. Disintegrate any $P  \in \RC(\rho)$ by writing $P(d\omega,ds) = \rho(d\omega)P(\omega,ds)$, and note that the above equation is equivalent to
\[
\int_{\Omega_{\mathrm{input}} \times [0,T]}P(d\omega,du)\psi(\omega)\psi_t(\omega)\phi_t(u) = \int_{\Omega_{\mathrm{input}}}\rho(d\omega)\psi(\omega)\psi_t(\omega)\int_{[0,T]}P(\omega,du)\phi_t(u).
\]
This is clearly a convex constraint on $P$, and we conclude that $\RC(\rho)$ is convex.

\textbf{$\bm{\RC^+(\rho)}$ is contained in the closure of $\bm{\RC_0(\rho)}$:}
Let $P \in \RC^+(\rho)$, and consider the process $H_t = 1_{\{\tau \le t\}}$ defined on $\Omega_{\mathrm{input}} \times [0,T]$. Note that $H$ is $\FF^{\tau}_+$-adapted, because $\{H_t = 1\} = \{\tau \le t\} \in \F^\tau_{t+}$. 
As a first approximation, note that the c\`adl\`ag process $H^n_t = 1_{\{(\tau + 1/n) \wedge T \le t\}}$ is $\FF^{\tau}$-adapted and converges almost surely to $H$ in the Skorohod topology. Because $\Phi(H^n) = (\tau + 1/n) \wedge T \rightarrow \tau = \Phi(H)$ a.s. by Lemma \ref{le:Phi-continuous}, we may henceforth assume that $H$ is in fact $\FF^{\tau}$-adapted.

Next, by a routine approximation we may find a sequence of c\`adl\`ag $\FF^{\tau}$-adapted processes $H^n$ converging almost surely to $H$ and of the form
\[
H^n_t = \sum_{k=1}^Kh^n_k1_{[t^n_k,t^n_{k+1})}(t),
\]
where $0 < t^n_1 < t^n_2 < \ldots < t^n_K = T < t^n_{K+1}$, and where $h^n_k \ge 0$ is an $\F^\tau_{t^n_k}$-measurable random variables. Replacing $h^n_k$ by $\max_{j=1,\ldots,k}h^n_j$ does not change the value of $\Phi(H^n)$, which again converges almost surely to $\Phi(H) = \tau$ by Lemma \ref{le:Phi-continuous}, which applies because $H \in \mathcal{S}$ a.s.
As a final approximation, let $\widehat{H}^n_t = H^n_t + t/n$; this last approximation accounts for the fact that $H^n$ may not belong a.s. to $\mathcal{S}$, while $\widehat{H}$ does. Note that $|H^n_t - \widehat{H}^n_t| \rightarrow 0$ uniformly in $t$, so $\lim_n\Phi(\widehat{H}^n) = \lim_n\Phi(H^n) = \tau$.
With these approximations, we may then assume henceforth that $H$ itself is increasing and of the form 
\[
H_t = \sum_{k=1}^Kh_k1_{[t_k,t_{k+1})}(t) + \frac{t}{n},
\]
where $0 < t_1 < t_2 < \ldots < t_K = T < t_{K+1}$, and where $h_k$ is $\F^\tau_{t_k}$-measurable for each $k$.

Now define $S_t$ for $t \in [0,T]$ by
\[
S_t = (B_{\cdot \wedge t}, W_{\cdot \wedge t}, \mu^t),
\]
where $m^t$ was defined just before the start of the proof. Then $S=(S_t)_{t \in [0,T]}$ is a continuous $\FF^{B,W,\mu}$-adapted process, with values in $\Omega_{\mathrm{input}}$. In fact, $\F^{B,W,\mu}_t = \sigma(S_t) = \sigma(S_{\cdot \wedge t})$ for every $t$.
For every $k=1,\ldots,K$, note that $(h_1,\ldots,h_k)$ is conditionally independent of $S$ given $S_{\cdot \wedge t_k}$. It follows from Proposition \ref{pr:density-discrete} that there exists a sequence of continuous functions $g^n_k : \Omega_{\mathrm{input}} \rightarrow \R_+$ such that $g^n_k(S)$ is $\sigma(S_{t_k})=\F^{B,W,\mu}_{t_k}$-measurable for each $k$ and
\[
(S,g^n_1(S),\ldots,g^n_K(S)) \Rightarrow (S,h_1,\ldots,h_K), \text{ in } \Omega_{\mathrm{input}} \times \R_+^K, \text{ as } n \rightarrow \infty.
\]
Now define
\[
H^n_t = H^n_t(S) = \sum_{k=1}^Kg^n_k(S)1_{[t_k,t_{k+1})}(t) + \frac{t}{n}.
\]
It follows that $(S,H^n) \Rightarrow (S,H)$, and because $H$ belongs almost surely to the set $\SS$ of Lemma \ref{le:Phi-continuous} we have $(S,\Phi(H^n)) \Rightarrow (S,\Phi(H))=(S,\tau)$.
Now let $\tilde{g}^n_k(s) = \max_{j=1,\ldots,k}g^n_j(s)$, and define
\[
\widetilde{H}^n_t = \widetilde{H}^n_t(S) = \sum_{k=1}^K\tilde{g}^n_k(S)1_{[t_k,t_{k+1})}(t).
\]
Then $\Phi(\widetilde{H}^n) = \Phi(H^n)$ almost surely, so again we have $(S,\Phi(\widetilde{H}^n)) \Rightarrow (S,\tau)$. Now, because each $\tilde{g}^n_k$ is continuous, we may view $\widetilde{H}(\cdot)$ as a continuous map from $\Omega_{\mathrm{input}}$ to the subset $\SS$ of $D$ defined in Lemma \ref{le:Phi-continuous}. Hence, Lemma \ref{le:Phi-continuous} ensures that $\Phi(\widetilde{H}^n(\cdot))$ is a continuous map from $\Omega_{\mathrm{input}}$ to $[0,T]$. The law of $(S,\Phi(\widetilde{H}^n))$ thus belongs to $\RC_0(\rho)$.
\hfill\qedsymbol

\subsection*{Proof of Proposition \ref{pr:strongisweak}}
With Corollary \ref{co:same-supremum} in hand, we are now ready to prove Proposition \ref{pr:strongisweak}.
It is readily checked that $P$ satisfies property (2) of Definition \ref{def:weakMFE} of a weak MFE. As $\tau^*$ is a $\FF^{B,W}$-stopping time, $\F^\tau_{t+}$ is contained in the $P$-completion of $\F^{B,W,\mu}_t$, and the compatibility property (3) holds easily (noting that compatibility is not sensitive to the completion of filtrations). To prove (1) is slightly more involved: Clearly $P \circ (B,W)^{-1} = \W^2$, where as usual $\W$ denotes Wiener measure. Note also that if $g_t : \C \rightarrow \R$ is bounded and $\F^{B,W}_t$-measurable, then
\[
\E[g_t(B,W) | B] = \int_{\C} g_t(B,w)\,\W(dw) = \E[g_t(B,W) | \F^B_t], \ a.s.
\]
Thus, since $\tau$ is a.s. $(B,W)$-measurable, if $t \in [0,T]$ and $C \in \F^{W,\tau}_t$ then
\[
\mu(C) = P((W,\tau) \in C | B) = P((W,\tau) \in C | \F^B_t), \ P-a.s.
\]
This shows that (the completion of) $\F^\mu_t$ is contained in (the completion of) $\F^B_t$. Hence, under $P$, the completion of $\F^{B,W,\mu,\tau}_t$ is contained in that of $\F^{B,W}_t$, which proves property (1). The weak fixed point condition (5) holds because $\mu$ is $B$-measurable.
Finally, the optimality condition (4) follows from Corollary \ref{co:same-supremum}.
\hfill \qedsymbol

\subsection{A shortcut to compatibility}
As a final preparatory step, before proving the main results we state one last useful lemma. It allows us to check a much simpler criterion in place of the compatibility property (3) of Definition \ref{def:weakMFE}, which does not behave too well under limits. In fact, this lemma is precisely the reason we work with the conditional \emph{joint} law of $(W,\tau)$ and not just $\tau$ itself.

\begin{lemma} \label{le:key-compatibility-test}
Suppose $P \in \P(\Omega)$ satisfies properties (2) and (5) of Definition \ref{def:weakMFE}. Suppose also that  $\F^{B,\mu}_T \vee \F^{W,\tau}_t$ is independent of $\sigma\{W_s - W_t : s \in [t,T]\}$, for every $t \in [0,T)$. Then $P$ satisfies property (3) of Definition \ref{def:weakMFE}; that is, $\F^\tau_{t+}$ is conditionally independent of $\F^{B,W,\mu}_T$ given $\F^{B,W,\mu}_{t+}$, for every $t \in [0,T]$.
\end{lemma}
\begin{proof}
According to the final claim of Theorem \ref{th:nonradomizeddensity2}, it suffices to check that $\F^\tau_{t}$ is conditionally independent of $\F^{B,W,\mu}_T$ given $\F^{B,W,\mu}_{t}$, for every $t \in [0,T]$.
Fix bounded functions $f_t$, $g_T$, $g_t$, $h^+$, and $h_t$, such that $f_t(\tau)$ is $\F^\tau_t$-measurable, $g_T(B,\mu)$ is $\F^{B,\mu}_T$-measurable, $g_t(B,\mu)$ is $\F^{B,\mu}_t$-measurable, $h^+(W)$ is $\sigma\{W_s - W_t : s \in [t,T]\}$-measurable, and $h_t(W)$ is $\F^W_t$-measurable. Compute
\begin{align*}
\E&\left[f_t(\tau)g_T(B,\mu)g_t(B,\mu)h^+(W)h_t(W)\right] \\
	&= \E\left[f_t(\tau)g_T(B,\mu)g_t(B,\mu)h_t(W)\right]\E[h^+(W)] \\
	&= \E\left[g_t(B,\mu)g_T(B,\mu)\int f_t(s)h_t(w)\mu(dw,ds)\right]\E[h^+(W)]  \\
	&= \E\left[g_t(B,\mu)\E\left[\left.g_T(B,\mu)\right|\F^{B,\mu}_t\right]\int f_t(s)h_t(w)\mu(dw,ds)\right]\E[h^+(W)] \\
	&= \E\left[g_t(B,\mu)h_t(W)f_t(\tau)\E\left[\left.g_T(B,\mu)\right|\F^{B,\mu}_t\right]\right]\E[h^+(W)] \\
	&= \E\left[g_t(B,\mu)h_t(W)f_t(\tau)\E\left[\left.g_T(B,\mu)h^+(W)\right|\F^{B,W,\mu}_t\right]\right].
\end{align*}
The first step used the assumed independence, whereas the second and fourth used the fixed point property $\mu = P((W,\tau) \in \cdot | B,\mu)$. The third step used the fact that $\int \phi\,d\mu$ is $\F^\mu_t$-measurable (and thus $\F^{B,\mu}_t$-measurable) for every bounded $\F^{W,\tau}_t$-measurable function $\phi$ on $\C \times [0,T]$. Finally, the last step used the easy identity
\[
\E\left[\left.g_T(B,\mu)h^+(W)\right|\F^{B,W,\mu}_t\right] = \E\left[\left.g_T(B,\mu)\right|\F^{B,\mu}_t\right]\E[h^+(W)].
\]
\end{proof}

\section{Proofs of limit theorems}
This section is devoted to the proofs of Theorems \ref{th:mainlimit} and \ref{th:converselimit}.
At this point it may be useful to recall the notations in these theorems. In particular, take care to distinguish the empirical measure of stopping times $\overline{\mu}(\vec{\tau}^{\,n})$, defined in \eqref{def:empirical-measure}, from the \emph{joint} empirical measure $\widehat{\mu}(\vec{\tau}^{\,n})$, defined in \eqref{def:empirical-measure-full}.

\subsection{Proof of Theorem \ref{th:mainlimit}}

Abbreviate $\widehat{\mu}^n = \widehat{\mu}^n(\vec{\tau}^{\,n})$. Note first that $\PP \circ (B,W^k)^{-1} = \W^2$ for all $k$, so the $\C^2$-marginal of $P_n$ does not depend on $n$.
Clearly, the $[0,T]$-marginal sequence $(P_n \circ \tau^{-1})_{n=1}^\infty$ is tight because $[0,T]$ is compact. To show that the marginal sequence $(P_n \circ (\widehat{\mu}^n)^{-1})_{n=1}^\infty$ is tight, it suffices to show that the sequence of mean measures $\E^{P_n}[\widehat{\mu}^n(\cdot)] \in \P(\W \times [0,T])$ is tight (c.f. the proof of \cite[Proposition 2.2]{sznitman}). But this follows from the observation that the first marginal of $\E^{P_n}[\widehat{\mu}^n(\cdot)]$ is the Wiener measure for each $n$.
As each of marginal sequences is tight, the sequence $(P_n)_{n=1}^\infty \subset \P(\Omega)$ is tight. Let $P$ be any limit point of $P_n$, and relabel the subsequence to assume $P_n \rightarrow P$. We check that $P$ satisfies the five defining properties of a weak MFE.

\textbf{Proof of (1):} 
First, note that
\begin{align*}
P \circ (B,W)^{-1} &= \lim_{n\rightarrow\infty}\frac{1}{n}\sum_{k=1}^n\PP \circ (B,W^k)^{-1} = \W^2.
\end{align*}
We next prove that $(B,W)$ is a Wiener process with respect to $\FF^{B,W,\mu,\tau}_+$, or equivalently with respect to $\FF^{B,W,\mu,\tau}$. Fix $t \in [0,T)$. Let $f_t$, $g_t$, $h_t$, and $h^+$ be bounded continuous functions on $[0,T]$, $\P(\C \times [0,T])$, $\C^2$, and $\C^2$, respectively. Assume  $f_t$ is $\F^\tau_{t}$-measurable, $g_t$ is $\F^\mu_t$-measurable, $h_t$ is $\F^{B,W}_t$-measurable, and $h^+$ is $\sigma\{(B_s-B_t,W_s-W_t) : s \in [t,T]\}$-measurable. Then, because $(B,W^1,\ldots,W^n)$ are $\FF^{B,W^1,\ldots,W^n}_+$-Wiener processes,
\begin{align*}
\E^P\left[f_t(\tau)g_t(\mu)h_t(B,W)h^+(B,W)\right] &= \lim_{n\rightarrow\infty}\frac{1}{n}\sum_{i=1}^n\E^{\PP}\left[f_t(\tau^n_i)g_t(\widehat{\mu}^n)h_t(B,W^i)h^+(B,W^i)\right] \\
	&= \lim_{n\rightarrow\infty}\frac{1}{n}\sum_{i=1}^n\E^{\PP}\left[f_t(\tau^n_i)g_t(\widehat{\mu}^n)h_t(B,W^i)\right]\E^{\PP}\left[h^+(B,W^i)\right] \\
	&= \E^P\left[f_t(\tau)g_t(\mu)h_t(B,W)\right]\E^P\left[h^+(B,W)\right].
\end{align*}
This shows that $\sigma\{(B_s-B_t,W_s-W_t) : s \in [t,T]\}$ is independent of $\F^{B,W,\mu,\tau}_t$, under $P$.

\textbf{Proof of (2):} To show that $(B,\mu)$ and $W$ are independent under $P$ is straightforward: For bounded continuous functions $f : \C \times \P(\C \times [0,T])  \rightarrow \R$ and $g : \C \rightarrow \R$, the law of large numbers yields
\begin{align*}
\E^P&\left[f(B,\mu)g(W)\right] - \E^P[f(B,\mu)]]\E^P[g(W)] \\
	&= \lim_{n\rightarrow\infty}\frac{1}{n}\sum_{i=1}^n\E^{\PP}\left[f(B,\widehat{\mu}^n)g(W^i)\right] - \E^{\PP}\left[f(B,\widehat{\mu}^n)\right]\frac{1}{n}\sum_{i=1}^n\E^{\PP}\left[g(W^i)\right] \\
	&= \lim_{n\rightarrow\infty}\E^{\PP}\left[f(B,\widehat{\mu}^n)\left(\frac{1}{n}\sum_{i=1}^ng(W^i) - \int g\,d\W\right)\right] \\
	&= 0,
\end{align*}
since $W^i$ are i.i.d. with law $\W$ under $\PP$.

\textbf{Proof of (5):}
The proof of the fixed point condition (5) is also straightforward. Let $f$ and $g$ be bounded continuous functions on $\C \times \P(\C \times [0,T])$ and $\C \times [0,T]$, respectively, and notice that
\begin{align*}
\E^P\left[f(B,\mu)g(W,\tau)\right] &= \lim_{n\rightarrow\infty}\frac{1}{n}\sum_{i=1}^n\E^{\PP}\left[f(B,\widehat{\mu}^n)g(W^i,\tau^n_i)\right] \\
	&= \lim_{n\rightarrow\infty}\E^{\PP}\left[f(B,\widehat{\mu}^n)\int g\,d\widehat{\mu}^n\right] \\
	&= \E^P\left[f(B,\mu)\int g\,d\mu\right].
\end{align*}

\textbf{Proof of (3):} 
Because we have established properties (2) and (5), Lemma \ref{le:key-compatibility-test} will immediately yield (3) once we can show that $\F^{B,\mu}_T \vee \F^{W,\tau}_t$ is independent of $\sigma\{W_s - W_t : s \in [t,T]\}$, for every $t \in [0,T)$.
Fix $t \in [0,T)$. Fix bounded continuous functions $f$ on $\C \times \P(\C \times [0,T])$, $g_t$ on $[0,T]$, $h_t$ on $\C$, and $h^+$ on $\C$. Assume $f$ is uniformly continuous (thus $\F^{B,\mu}_T$-measurable), $g$ is $\F^{\tau}_t$-measurable, $h_t$ is $\F^W_t$-measurable, and $h^+$  is $\sigma\{W_s - W_t : s \in [t,T]\}$-measurable. Define
\[
\widehat{\mu}^{n,-i} := \frac{1}{n-1}\sum_{k \neq i}\delta_{(W^k,\tau^n_k)} = \frac{n}{n-1}\widehat{\mu}^n - \frac{1}{n-1}\delta_{(W^k,\tau^n_k)},
\]
and note that $\|\widehat{\mu}^{n,-i} - \widehat{\mu}^n\|_{TV} \le 2/(n-1)$ a.s. The total variation topology is finer than weak convegence, and so
\[
|f(B,\widehat{\mu}^n) - f(B,\widehat{\mu}^{n,-i})| \rightarrow 0,
\]
in $L^\infty$, uniformly in $i$. Now, since $\sigma\{W^i_s - W^i_t : s \in [t,T]\}$ is independent of $\F^{W^i}_t \vee \F^{B,(W^k)_{k \neq i}}_T$, we have
\begin{align*}
\E^P\left[f(B,\mu)g_t(\tau)h_t(W)h^+(W)\right] &= \lim_{n\rightarrow\infty}\frac{1}{n}\sum_{i=1}^n\E^{\PP}\left[f(B,\widehat{\mu}^n)g_t(\tau^n_i)h_t(W^i)h^+(W^i)\right] \\
	&= \lim_{n\rightarrow\infty}\frac{1}{n}\sum_{i=1}^n\E^{\PP}\left[f(B,\widehat{\mu}^{n,-i})g_t(\tau^n_i)h_t(W^i)h^+(W^i)\right] \\
	&= \lim_{n\rightarrow\infty}\frac{1}{n}\sum_{i=1}^n\E^{\PP}\left[f(B,\widehat{\mu}^{n,-i})g_t(\tau^n_i)h_t(W^i)\right]\E^{\PP}[h^+(W^i)] \\
	&= \lim_{n\rightarrow\infty}\frac{1}{n}\sum_{i=1}^n\E^{\PP}\left[f(B,\widehat{\mu}^n)g_t(\tau^n_i)h_t(W^i)\right]\E^{\PP}[h^+(W^i)] \\
	&= \E^P\left[f(B,\mu)g_t(\tau)h_t(W)\right]\E^P[h^+(W)].
\end{align*}
This implies $\F^{B,\mu}_T \vee \F^{W,\tau}_t$ is independent of $\sigma\{W_s - W_t : s \in [t,T]\}$, under $P$.

\textbf{Proof of (4):}
It remains to prove that the optimality condition holds in the limit. Recall that $\mu^\tau(\cdot) = \mu(\C \times \cdot)$ denotes the $[0,T]$-marginal of $\mu$.
By Corollary \ref{co:same-supremum}, it suffices to show
\[
\E^P[F(B,W,\mu^\tau,\tau)] \ge \E^P[F(B,W,\mu^\tau,\sigma)]
\]
for every $\FF^{B,W,\mu}$-stopping time $\sigma$ on $\Omega$ of the form $\sigma = \hat{\sigma}(B,W,\mu)$, where $\hat{\sigma} : \C^2 \times \P(\C \times [0,T]) \rightarrow [0,T]$ is continuous. Fix such a stopping time. 
For the $n$-player game define
\[
\sigma_i = \hat{\sigma}(B,W^i,\widehat{\mu}^n(\vec{\tau}^{\,n})).
\]
Then $\sigma_i$ is a $\FF^{B,W^1,\ldots,W^n}$-stopping time. Recall that $\overline{\mu}^n$ denotes the $[0,T]$-marginal of the joint empirical measure $\widehat{\mu}^n$. The Nash property implies
\begin{align*}
\E^P[F(B,W,\mu,\tau)] &= \lim_{n\rightarrow\infty}\frac{1}{n}\sum_{i=1}^n\E^{\PP}\left[F(B,W^i,\overline{\mu}^n(\vec{\tau}^{\,n}),\tau^n_i) \right] \\
	&\ge \limsup_{n\rightarrow\infty}\frac{1}{n}\sum_{i=1}^n\E^{\PP}\left[F(B,W^i,\overline{\mu}^n(\vec{\tau}^{\,n,-i},\sigma_i),\sigma_i) \right] \\
	&= \limsup_{n\rightarrow\infty}\frac{1}{n}\sum_{i=1}^n\E^{\PP}\left[F(B,W^i,\overline{\mu}^n(\vec{\tau}^{\,n}),\hat{\sigma}(B,W^i,\widehat{\mu}^n(\vec{\tau}^{\,n}))) \right] \\
	&= \E^P[F(B,W,\mu^\tau,\hat{\sigma}(B,W,\mu))].
\end{align*}
Indeed, the equality in the third line follows from the easy estimate $\|\overline{\mu}^n(\vec{\tau}^{\,n,-i},\sigma_i) - \overline{\mu}^n(\vec{\tau}^{\,n})\|_{TV} \le 2/n$, where $\|m\|_{TV} = \sup_{|f| \le 1}\int f\,dm$ denotes total variation, and also from the continuity of $F=F(b,w,m,t)$ in $m$ ensured by assumption \textbf{D}. Both the first and last lines use Lemma \ref{le:jacodmemin} to deal with the potential discontinuity of $F$ in $(B,W)$, with the last step using crucially the continuity of $\hat{\sigma}$.
\hfill \qedsymbol

\subsection{Proof of Theorem \ref{th:converselimit}}

Let $P \in \P(\Omega)$ be a weak MFE. Construct, on some alternative probability space $(\widetilde{\Omega},\widetilde{\F},\widetilde{\PP})$, a $\C \times \P(\C \times [0,T])$-valued random variable $(B,\mu)$ with law $P \circ (B,\mu)^{-1}$ and a sequence of random variables $(W^i,\tau^i)$, which are conditionally independent given $(B,\mu)$ and have common conditional law $\mu$. There is some abuse of notation here, as $(B,\mu)$ is used both for the new random variable and for the random variable defined on the canonical space $\Omega$, but this should cause no confusion as we work exclusively on $(\widetilde{\Omega},\widetilde{\F},\widetilde{\PP})$ in this proof.

The law of $(B,\mu,W^i,\tau^i)$ is precisely $P$, for each $i$. As usual, let $\vec{\tau}^{\,n} = (\tau^1,\ldots,\tau^n)$, and for $t^1,\ldots,t^n \in [0,T]$ define the empirical measures (now on $\widetilde{\Omega}$)
\[
\overline{\mu}^n(t^1,\ldots,t^n) = \frac{1}{n}\sum_{i=1}^n\delta_{t^i}, \quad\quad \widehat{\mu}^n(t^1,\ldots,t^n) = \frac{1}{n}\sum_{i=1}^n\delta_{(W^i,t^i)}.
\]
Define
\begin{align*}
\epsilon_n := \left\{\sup_{\sigma \in \SS_n}\E\left[F(B,W^1,\overline{\mu}^n(\vec{\tau}^{\,n,-1},\sigma),\sigma)\right] - \E\left[F(B,W^1,\overline{\mu}^n(\vec{\tau}^{\,n}),\tau^1)\right]\right\}^+,
\end{align*}
where $\SS_n$ is the set of $\FF^{B,W^1,\ldots,W^n}$-stopping times (defined on $\widetilde{\Omega}$). By symmetry
\begin{align}
\sup_{\sigma \in \SS_n}\E\left[F(B,W^k,\overline{\mu}^n(\vec{\tau}^{\,n,-k},\sigma),\sigma)\right] \le \epsilon_n + \E\left[F(B,W^k,\overline{\mu}^n(\vec{\tau}^{\,n}),\tau^k)\right],
\end{align}
for every $k=1,\ldots,n$. We first show that $\epsilon_n \rightarrow 0$. Indeed, 
\begin{align*}
\epsilon_n &\le \E\left[\sup_{t \in [0,T]}\left|F(B,W^1,\overline{\mu}^n(\vec{\tau}^{\,n,-1},t),t) - F(B,W^1,\mu,t)\right|\right] \\
	&\quad\quad + \left\{\sup_{\sigma \in \SS_\infty}\E\left[F(B,W^1,\mu,\sigma)\right] - \E\left[F(B,W^1,\overline{\mu}^n(\vec{\tau}^{\,n}),\tau^1)\right]\right\}^+,
\end{align*}
where $\SS_\infty = \cup_{n\ge 1}\SS_n$.
We claim that the term on the first line converges to zero. Indeed, $\|\overline{\mu}^n(\vec{\tau}^{\,n,-1},t)-\overline{\mu}^n(\vec{\tau}^{\,n})\| \le 2/n$, and $\overline{\mu}^n(\vec{\tau}^{\,n}) \rightarrow \mu$ weakly a.s. by the (conditional) law of large numbers. Conclude using the assumption that $\P([0,T]) \times [0,T]) \ni (m,t) \mapsto F(b,w,m,t)$ is (uniformly) continuous for each fixed $(b,w)$.
The second term also tends to zero, because
\[
\E\left[F(B,W^1,\overline{\mu}^n(\vec{\tau}^{\,n}),\tau^1)\right] \rightarrow \E\left[F(B,W^1,\mu,\tau^1)\right] \ge \sup_{\sigma \in \SS_\infty}\E\left[F(B,W^1,\mu,\sigma)\right].
\]
Indeed, to prove the last inequality, note that for any $\sigma \in \SS_\infty$ we can easily check that $\F^\sigma_t$ is conditionally independent of $\F^{B,W^1,\mu}_T$ given $\F^{B,W^1,\mu}_t$, for every $t$. Because $(B,W^1,\mu,\tau^1)$ has law $P$, which is a weak MFE, the optimality condition (4) of Definition \ref{def:weakMFE} provides the desired inequality.

It may appear that we have shown that $(\tau^1,\ldots,\tau^n)$ form an $\epsilon_n$-Nash equilibrium for the $n$-player game, with $\epsilon_n \rightarrow 0$, but this is not accurate. The stopping times $\tau^i$ are not stopping times relative to $\FF^{B,W^1,\ldots,W^n}$, but rather to a larger filtration. This necessitates one more approximation, using a straightforward extension of Theorem \ref{th:nonradomizeddensity2} to deal with vectors of stopping times as opposed to single stopping times; the proof of this extension is exactly the same but notationally more cumbersome. Note that $\F^{\tau^1,\ldots,\tau^n}_t$ is conditionally independent of $\F^{B,W^1,\ldots,W^n}_T$ given $\F^{B,W^1,\ldots,W^n,\mu}_t$, for every $t$, simply because $(B,W^1,\ldots,W^n)$ is a $\FF^{B,W^1,\ldots,W^n,\tau^1,\ldots,\tau^n}$-Wiener process. Hence, using the aforementioned extension of Theorem \ref{th:nonradomizeddensity2}, we may find $\FF^{B,W^1,\ldots,W^n}$-stopping times, $\tau^1_k,\ldots,\tau^n_k$, such that
\[
(B,W^1,\ldots,W^n,\tau^1_k,\ldots,\tau^n_k) \Rightarrow (B,W^1,\ldots,W^n,\tau^1,\ldots,\tau^n),
\]
as $k\rightarrow\infty$. Let $\vec{\tau}^{\,n,k} = (\tau^1_k,\ldots,\tau^n_k)$, and define
\begin{align*}
\epsilon^k_n := \max_{i=1,\ldots,n}\left\{\sup_{\sigma \in \SS_n}\E\left[F(B,W^i,\overline{\mu}^n(\vec{\tau}^{\,n,k,-i},\sigma),\sigma)\right] - \E\left[F(B,W^i,\overline{\mu}^n(\vec{\tau}^{\,n,k}),\tau^i_k)\right]\right\}^+.
\end{align*}
For a fixed $n$, we can argue that $\lim_{k\rightarrow\infty}\epsilon^k_n = \epsilon_n$. Indeed, this follows from the observation that
\begin{align*}
\lim_{k\rightarrow\infty}\max_{i=1,\ldots,n}\left|\E\left[F(B,W^i,\overline{\mu}^n(\vec{\tau}^{\,n,k}),\tau^i_k)\right] - \E\left[F(B,W^i,\overline{\mu}^n(\vec{\tau}^{\,n}),\tau^i)\right]\right| = 0,
\end{align*}
by construction, and
\begin{align*}
\lim_{k\rightarrow\infty}\max_{i=1,\ldots,n}\E\left[\sup_{t \in [0,T]}\left|F(B,W^i,\overline{\mu}^n(\vec{\tau}^{\,n,k,-i},t),t) - F(B,W^i,\overline{\mu}^n(\vec{\tau}^{\,n,-i},t),t)\right|\right] =0,
\end{align*}
because $m \mapsto F(b,w,m,t)$ is continuous in $m$, \emph{uniformly in $t$}, for almost every fixed $(b,w)$.

In summary, we may find $k_n \rightarrow \infty$ such that
\[
\frac{1}{n}\sum_{i=1}^n\text{Law}(B,W^i,\widehat{\mu}^n(\vec{\tau}^{\,n,k_n}),\tau^i_{k_n}) \rightarrow \text{Law}(B,W^1,\mu,\tau^1),
\]
as well as $\epsilon^{k_n}_n \downarrow 0$ and
\begin{align*}
\max_{i=1,\ldots,n}\sup_{\sigma\in\SS_n}\E\left[F(B,W^i,\overline{\mu}^n(\vec{\tau}^{\,n,k_n,-i},\sigma),\sigma)\right] &\le 
\E\left[F(B,W^i,\overline{\mu}^n(\vec{\tau}^{\,n,k_n}),\tau^i_{k_n})\right] + \epsilon^{k_n}_n.
\end{align*}

\section{Existence under continuity assumption \textbf{D}}
This section is devoted to the proof of Theorem \ref{th:existence}. First, we prove existence of a weak MFE (more precisely, a strong MFE with weak stopping time) in the case that the time set and the common noise range space are finite. Then, we take weak limits. We introduce the following discretization.

For each positive integer $n$, let $t^n_i = i2^{-n}T$ for $i=0,\ldots,2^n$. Choose partitions $\pi^n$ of $\R$ into $n$ measurable sets of strictly positive Lebesgue measure, in such a way that $\pi^{n+1}$ is a refinement of $\pi^n$ for each $n$ and the union $\cup_{n\ge 1}\pi^n$ generates the entire Borel $\sigma$-field. We will define a sub-filtration of $\FF^B$ according to which sets of $\pi^n$ contains the increments  $B_{t^n_{i+1}}-B_{t^n_i}$. Precisely, define the $\sigma$-field
\begin{align*}
\G^n_{t^n_k} &= \sigma\left\{\{B_{t^n_i}-B_{t^n_{i-1}} \in C\} : C \in \pi^n, \ i=1,\ldots,k\right\}, \text{ for } k=1,\ldots,2^n, \\
\G^n_0 &= \sigma\{\{B_0 \in C\} : C \in \pi^n\}.
\end{align*}
Additionally, for $t \in [0,T]$, define $\lfloor t\rfloor_n = \max\{t^n_k : k=0,\ldots,2^n, \ t^n_k \le t\}$, and let $\G^n_t = \G^n_{\lfloor t\rfloor_n}$. Then $\GG^n = (\G^n_t)_{t \in [0,T]}$ defines a filtration, and $\G^n_t$ is finite for each $n,t$. Moreover,
\[
\F^B_t = \sigma\left(\bigcup_{n=1}^\infty\G^n_t\right), \text{ for } t \in [0,T].
\]
By construction, $\W(C) > 0$ for every nonempty set $C \in \G^n_T$.

Let $\M_n$ denote the set of functions $\overline{m} : \C \rightarrow \P([0,T])$ such that, for each $t \in [0,T]$ the map $b \mapsto \overline{m}(b)[0,t]$ is $\G^n_t$-measurable. Of course, since $\G^n_T$ is finite, a map $\overline{m} \in \M_n$ must be constant on each atom of $\G^n_T$. Endowed with the topology of pointwise convergence, $\M_n$ is easily seen to be compact, because $\P([0,T])$ is compact. Lastly, define $\A$  as the set of probability measures $Q \in \P(\C^2 \times [0,T])$ under which $B$ and $W$ are independent $\FF^{B,W,\tau}$-Wiener processes.

\begin{theorem} \label{th:discrete-existence}
For each $n$, there exist $\overline{m} \in \M_n$ and $Q \in \A$ satisfying the following:
\begin{enumerate}
\item $\overline{m}(B) = Q(\tau \in \cdot | \G^n_T)$.
\item The optimality condition holds,
\[
\E^Q[F(B,W,\overline{m}(B),\tau)] \ge \sup_{Q' \in \A}\E^{Q'}[F(B,W,\overline{m}(B),\tau)]
\]
\end{enumerate}
\end{theorem}
\begin{proof}
It follows from Theorem \ref{th:nonradomizeddensity2} that $\A$ is compact and convex; see Remark \ref{re:wiener-filtration}.
Define a map $\Phi_n$ from $\M_n$ to subsets of $\A$ by
\[
\Phi_n(\overline{m}) = \arg\max_{Q \in \A}\E^Q[F(B,W,\overline{m}(B),\tau)].
\]
The map $(\overline{m},Q) \mapsto \E^Q[F(B,W,\overline{m}(B),\tau)]$ is jointly continuous on $\M_n \times \A$, thanks to Lemma \ref{le:jacodmemin} and continuity of $F=F(b,w,m,t)$ in $(m,t)$. Thus, by Berge's theorem \cite[Theorem 17.31]{aliprantisborder}, $\Phi_n$ has closed graph and takes nonempty convex values. Define a map $\Psi_n : \M_n \rightarrow 2^{\M_n}$ by
\[
\Psi_n(\overline{m}) = \left\{Q(\tau \in \cdot | \G^n_T) : Q \in \Phi_n(\overline{m})\right\}.
\]
Because $Q \circ B^{-1} = \W$, for continuous bounded functions $f : [0,T] \rightarrow \R$ we may write
\[
\E^Q[f(\tau) | \G^n_T](b) = \sum_C \frac{\E^Q[f(\tau)1_{\{B \in C\}}]}{\W(C)}1_C(b), \text{ for } b \in \C,
\]
where the sum is over atoms of $\G^n_T$. For each such atom $C$, note that $\W(C) > 0$, and the map $Q \mapsto \E^Q[f(\tau)1_{\{B \in C\}}]$ is continuous on $\A$ by Lemma \ref{le:jacodmemin}. We may thus view $Q \mapsto Q(\tau \in \cdot | \G^n_T)$ as a continuous affine map from $\A$ to $\M_n$. Hence, the set-valued map $\Psi_n$ has closed graph, and its values are nonempty, convex, and compact. By Kakutani's theorem \cite[Corollary 17.55]{aliprantisborder}, it admits a fixed point; that is, there exists $\overline{m} \in \M_n$ such that $\overline{m} \in \Psi_n(\overline{m})$.
\end{proof}

\subsection*{Proof of Theorem \ref{th:existence}}
For each $n$, let $\overline{m}_n \in \M_n$ and $Q_n \in \A$ satisfy properties (1-2) of Theorem \ref{th:discrete-existence}. Define $\widehat{m}_n : \C \rightarrow \P(\C \times [0,T])$ by
\[
\widehat{m}_n(B) = Q_n((W,\tau) \in \cdot | \G^n_T).
\]
We note here that if $C \in \F^{W,\tau}_t$ for $t \in [0,T]$, then we have
\begin{align}
\widehat{m}_n(B) = Q_n((W,\tau) \in \cdot | \G^n_{\lceil t \rceil_n}), \label{def:hatq-sigmafield}
\end{align}
where $\lceil t \rceil_n = \min\{t^n_k : k=0,\ldots,2^n, \ t^n_k \ge t\}$. Indeed, this holds because $B$ is an $\FF^{B,W,\tau}$-Wiener process under $Q_n$, and because $\G^n_T$ is generated by $\G^n_{\lceil t \rceil_n}$ and the events $\{B_{t^n_i}-B_{t^n_{i-1}} \in C'\}$ for $C' \in \pi^n$ and for $t^n_{i-1} \ge t$.

Next, define $P_n \in \P(\Omega)$ by
\[
P_n = Q_n \circ (B,W,\widehat{m}_n(B),\tau)^{-1}.
\]
We claim that $(P_n)_{n=1}^\infty$ is tight and that every limit point is a weak MFE in the sense of Definition \ref{def:weakMFE}.
We begin with tightness, by showing that each marginal is tight. Note that the first marginal $P_n \circ (B,W)^{-1} = \W^2$ is clearly tight, as it is constant. Moreover, $P_n \circ \tau^{-1}$ is tight because $[0,T]$ is compact. To prove that $P_n \circ \mu^{-1}$ is tight, it suffices to show that the mean measures $\E^{P_n}[\mu(\cdot)]$ are tight (c.f. the proof of \cite[Proposition 2.2]{sznitman}). But, for any measurable set $C \subset \C \times [0,T]$,
\[
\E^{P_n}[\mu(C)] = \E^{Q_n}[\widehat{m}_n(B)(C)] = \E^{Q_n}[Q_n((W,\tau) \in C | \G^n_T)] = Q_n((W,\tau) \in C) = P_n((W,\tau) \in C).
\]
That is, the mean measure $\E^{P_n}[\mu(\cdot)]$ is precisely $P_n \circ (W,\tau)^{-1}$, which we already observed to be tight.

With tightness of $(P_n)_{n=1}^\infty$ established, fix a limit point $P$ and abuse notation by assuming $P_n \rightarrow P$. We will show that $P$ is a weak MFE by checking the five properties of Definition \ref{def:weakMFE}:

\textbf{Proof of (1):} Clearly $P \circ (B,W)^{-1} = \lim_{n\rightarrow\infty}Q_n \circ (B,W)^{-1} = \W^2$. We must show that $(B,W)$ are $\FF^{B,W,\mu,\tau}$-Wiener processes under $P$. Fix $N$ and fix $t \in \{t^N_0,\ldots,t^N_{2^N-1}\}$. Let $f_t$, $g_t$, $h_t$, and $h^+$ be bounded continuous functions on $[0,T]$, $\P(\C\times[0,T])$, $\C^2$, and $\C^2$, respectively. Assume $f_t$ is $\F^\tau_t$-measurable, $g_t$ is $\F^\mu_t$-measurable, $h_t$ is $\F^{B,W}_t$-measurable, and $h^+$ is $\sigma\{(B_s-B_t,W_s-W_t) : s \in [t,T]\}$-measurable. Since $\lceil t \rceil_n = t$ for $n \ge N$ and $\G^n_t \subset \F^B_t$, it follows from \eqref{def:hatq-sigmafield} that the map $b \mapsto \widehat{m}_n(b)(C)$ is $\F^B_t$-measurable for every $C \in \F^{W,\tau}_t$. This implies $\widehat{m}_n$ is $\F^B_t/\F^\mu_t$-measurable and, in particular, $g_t(\widehat{m}_n(B))$ is $\F^B_t$-measurable.  Thus, since $(B,W)$ is an $\FF^{B,W,\tau}$-Wiener process under $Q_n$,
\begin{align*}
\E^P[f_t(\tau)g_t(\mu)h_t(B,W)h^+(B,W)] &= \lim_{n\rightarrow\infty}\E^{Q_n}[f_t(\tau)g_t(\widehat{m}_n(B))h_t(B,W)h^+(B,W)] \\
	&= \lim_{n\rightarrow\infty}\E^{Q_n}[f_t(\tau)g_t(\widehat{m}_n(B))h_t(B,W)]\E^{Q_n}[h^+(B,W)] \\
	&= \E^P[f_t(\tau)g_t(\mu)h_t(B,W)]\E^P[h^+(B,W)].
\end{align*}
This is enough to conclude that $\sigma\{(B_s-B_t,W_s-W_t) : s \in [t,T]\}$ is independent of $\F^{B,W,\mu,\tau}_{t}$, under $P$. We have only shown this to be true for $t \in \cup_{N=1}^\infty\{t^N_0,\ldots,t^N_{2^N-1}\}$, but this it suffices to note that this set is dense in $[0,T]$.

\textbf{Proof of (2):} Because $B$ and $W$ are independent under $Q_n$, for bounded continuous functions $f$ on $\C \times \P(\C \times [0,T])$ and $g$ on $\C$ we have
\begin{align*}
\E^P[f(B,\mu)g(W)] &= \lim_{n\rightarrow\infty}\E^{Q_n}[f(B,\widehat{m}(B))g(W)] \\
	&= \lim_{n\rightarrow\infty}\E^{Q_n}[f(B,\widehat{m}(B))]\E^{Q_n}[g(W)]  \\
	&= \E^P[f(B,\mu)]\E^P[g(W)].
\end{align*}
Thus $(B,\mu)$ and $W$ are independent under $P$.

\textbf{Proof of (5):} Fix $N \ge 1$, and let $f$ be a bounded $\G^N_T$ measurable function on $\C$. Let $g$ and $h$ be continuous bounded functions on $\C \times [0,T]$ and $\P(\C \times [0,T])$, respectively. Noting that $f$ is $\G^n_T$-measurable for all $n \ge N$, and using Lemma \ref{le:jacodmemin} to deal with the discontinuity of $f$ in $b$, we have 
\begin{align*}
\E^P\left[f(B)h(\mu)g(W,\tau)\right] &= \lim_{n\rightarrow\infty}\E^{Q_n}\left[f(B)h(\widehat{m}_n(B))g(W,\tau)\right] \\
	&= \lim_{n\rightarrow\infty}\E^{Q_n}\left[f(B)h(\widehat{m}_n(B))\int g\,d\widehat{m}_n(B)\right] \\
	&= \E^P\left[f(B)h(\mu)\int g\,d\mu\right].
\end{align*}
This holds for each $N$ and each $\G^N_T$-measurable $f$. Since $\F^B_T = \sigma(\cup_{n \ge 1}\G^N_T)$, the same identity must hold for every $\F^B_T$-measurable $f$.

\textbf{Proof of (3):} 
Because we have established properties (2) and (5), Lemma \ref{le:key-compatibility-test} will yield (3) once we can show that $\F^{B,\mu}_T \vee \F^{W,\tau}_t$ is independent of $\sigma\{W_s - W_t : s \in [t,T]\}$, for every $t \in [0,T)$.
Fix $t \in [0,T)$. Fix bounded continuous functions $f$ on $\C \times \P(\C \times [0,T])$, $g_t$ on $[0,T]$, $h_t$ on $\C$, and $h^+$ on $\C$. Assume $g$ is $\F^{\tau}_t$-measurable, $h_t$ is $\F^W_t$-measurable, and $h^+$  is $\sigma\{W_s - W_t : s \in [t,T]\}$-measurable. Under $Q_n$, $(B,W)$ is a standard $\FF^{B,W,\tau}$-Wiener process, and it follows easily that $h^+(W)$ is independent of $\F^B_T \vee \F^{W,\tau}_t$. Thus, 
\begin{align*}
\E^P\left[f(B,\mu)g_t(\tau)h_t(W)h^+(W)\right] &= \lim_{n\rightarrow\infty}\E^{Q_n}\left[f(B,\widehat{m}_n(B))g_t(\tau)h_t(W)h^+(W)\right] \\
	&= \lim_{n\rightarrow\infty}\E^{Q_n}\left[f(B,\widehat{m}_n(B))g_t(\tau)h_t(W)\right]\E^{Q_n}[h^+(W)] \\
	&= \E^P\left[f(B,\mu)g_t(\tau)h_t(W)\right]\E^P[h^+(W)].
\end{align*}
This shows $\F^{B,\mu}_T \vee \F^{W,\tau}_{t}$ is independent of $\sigma\{W_s - W_t : s \in [t,T]\}$, under $P$.

\textbf{Proof of (4):}
By Corollary \ref{co:same-supremum}, it suffices to show
\[
\E^P[F(B,W,\mu^\tau,\tau)] \ge \E^P[F(B,W,\mu^\tau,\sigma)]
\]
for every $\FF^{B,W,\mu}$-stopping time $\sigma$ of the form $\sigma = \hat{\sigma}(B,W,\mu)$, where $\hat{\sigma} : \C^2 \times \P(\C \times [0,T]) \rightarrow [0,T]$ is continuous. Fix such a stopping time. Using the continuity of $\hat{\sigma}$, and using Lemma \ref{le:jacodmemin} to handle the discontinuity of $F$ in $(B,W)$, we have
\begin{align*}
\E^P&[F(B,W,\mu^\tau,\tau)] - \E^P[F(B,W,\mu^\tau,\sigma)] \\
	&= \lim_{n\rightarrow\infty}\E^{Q_n}[F(B,W,\overline{m}_n(B),\tau)] - \E^{Q_n}[F(B,W,\overline{m}_n(B),\sigma(B,W,\widehat{m}_n(B)))] \\
	&\ge 0,
\end{align*}
where we have finally used the optimality property of $\overline{m}_n$ from Theorem \ref{th:discrete-existence}. Indeed, the law
\[
Q'_n := Q_n \circ (B,W,\sigma(B,W,\widehat{m}_n(B)))^{-1}
\]
is easily seen to belong to $\A$, by noting that (as we saw in the proof of (1)) $\widehat{m}_n$ is adapted in the sense that $b \mapsto \widehat{m}_n(b)(C)$ is $\F^B_t$-measurable for $C \in \F^{W,\tau}_t$.
\hfill\qedsymbol

\appendix

\section{Proof of Proposition \ref{pr:density-discrete}} \label{se:adapted-density-proof}

To prove Theorem \ref{pr:density-discrete}, we need a preliminary results, borrowed from previous works of the authors. Recall that $\Rightarrow$ denotes convergence in law.

\begin{proposition}[Proposition C.1 of \cite{carmonadelaruelacker-mfgcommonnoise}] \label{pr:density-simple}
Let $X$ and $Y$ be random variables defined on a common probability space, taking values in some Polish spaces $E$ and $F$. If the law of $X$ is nonatomic, and if $F$ is (homeomorphic to) a convex subset of a locally convex space, then there exists a sequence of continuous functions $\phi_n : E \rightarrow F$ such that $(X,\phi_n(X)) \Rightarrow (X,Y)$.
\end{proposition}

Proposition \ref{pr:density-discrete} extends Proposition \ref{pr:density-simple} to the dynamic setting, and this is where the role of compatibility is the clearest. This is contained in the third author's PhD thesis \cite[Proposition 2.1.6]{lacker-thesis}, which itself was implicitly present in the proof of \cite[Lemma 3.11]{carmonadelaruelacker-mfgcommonnoise}, though we include the proof for the sake of completeness.

\subsection*{Proof of Proposition \ref{pr:density-discrete}}
The proof is an inductive application of Proposition \ref{pr:density-simple}. First, in light of the assumption that the law of $Y_1$ is nonatomic, Proposition \ref{pr:density-simple} allows us to find a sequence of continuous functions $h^j_1 : \Y \rightarrow \X$ such that $(Y_1,h^j_1(Y_1)) \Rightarrow (Y_1,X_1)$ as $j\rightarrow\infty$.  Let us show that in fact $(Z,h^j_1(Y_1))$ converges to $(Z,X_1)$.
Let $\phi : \Z \rightarrow \R$ be bounded and measurable, and let $\psi : \X \rightarrow \R$ be continuous. Note that $Z$ and $X_1$ are conditionally independent given $Y_1$, by assumption. Now use Lemma \ref{le:jacodmemin} to get
\begin{align*}
\lim_{j\rightarrow\infty}\E[\phi(Z)\psi(h^j_1(Y_1))] &= \lim_{j\rightarrow\infty}\E\left[\E\left[\left.\phi(Z)\right| Y_1\right]\psi(h^j_1(Y_1))\right]	\\
	&= \E\left[\E\left[\left.\phi(Z)\right| Y_1\right])\psi(X_1)\right] \\
	&= \E\left[\E\left[\left.\phi(Z)\right| Y_1\right] \E\left[\left.\psi(X_1)\right| Y_1\right]\right] \\
	&= \E\left[\E\left[\left.\phi(Z)\psi(X_1)\right| Y_1\right]\right] \\
	&= \E\left[\phi(Z)\psi(X_1)\right].
\end{align*}
The class of functions of the form $\Z \times \X \ni (z,x) \mapsto \phi(z)\psi(x)$, where $\phi$ and $\psi$ are as above, is convergence determining (see, e.g., \cite[Proposition 3.4.6(b)]{ethier-kurtz}), and we conclude that $(Z,h^j_1(Y_1)) \Rightarrow (Z,X_1)$.

We proceed inductively as follows. Abbreviate $Y^n := (Z_1,\ldots,Z_n)$ for each $n=1,\ldots,N$, noting $Y^N=Y$, and similarly define $X^n$. Suppose we are given $1 \le n < N$ and continuous functions  $g^j_k : \Y^k \rightarrow \X$, for $k \in \{1,\ldots,n\}$ and $j \ge 1$, satisfying
\begin{align}
\lim_{j\rightarrow\infty}(Z,g^j_1(Y^1),\ldots,g^j_n(Y^n)) = (Z,X^n), \label{pf:density1}
\end{align}
where convergence is in distribution, as usual.
We will show that there exist continuous functions $h^i_k : \Y^k \rightarrow \X$ for each $k \in \{1,\ldots,n+1\}$ and $i \ge 1$ such that 
\begin{align}
\lim_{i\rightarrow\infty}(Z,h^i_1(Y^1),\ldots,h^i_{n+1}(Y^{n+1})) = (Z,X_1,\ldots,X_{n+1}). \label{pf:density2}
\end{align}
By Proposition \ref{pr:density-simple} there exists a sequence of continuous functions $\hat{g}^j : (\Y^{n+1} \times \X^n) \rightarrow \X$ such that
\begin{align}
\lim_{j\rightarrow\infty}(Y^{n+1},X^n,\hat{g}^j(Y^{n+1},X^n)) = (Y^{n+1},X^n,X_{n+1}) = (Y^{n+1},X^{n+1}). \label{pf:density2.5}
\end{align}
We claim now that
\begin{align}
\lim_{j\rightarrow\infty}(Z,X^n,\hat{g}^j(Y^{n+1},X^n)) = (Z,X^n,X_{n+1}). \label{pf:density3}
\end{align}
Indeed, let $\phi$, $\psi_n$, and $\psi$ be bounded measurable functions on $\Z$, $\X^n$, and $\X$, respectively, with $\psi_n$ and $\psi$ continuous. Use the conditional independence of $Z$ and $(Y^{n+1},X^{n+1})$ given $Y^{n+1}$ along with \eqref{pf:density2.5} and Lemma \ref{le:jacodmemin} to get
\begin{align*}
\lim_{j\rightarrow\infty}\E[\phi(Z)\psi_n(X^n)\psi(\hat{g}^j(Y^{n+1},X^n))]&= \lim_{j\rightarrow\infty}\E\left[\E\left[\left.\phi(Z)\right| Y^{n+1}\right]\psi_n(X^n)\psi(\hat{g}^j(Y^{n+1},X^n))\right]	 \\
	&= \E\left[\E\left[\left.\phi(Z)\right| Y^{n+1}\right]\psi_n(X^n)\psi(X_{n+1})\right] \\
	&= \E\left[\E\left[\left.\phi(Z)\right| Y^{n+1}\right]\E\left[\left.\psi_n(X^n)\psi(X_{n+1})\right| Y^{n+1}\right]\right] \\
	&= \E\left[\E\left[\left.\phi(Z) \psi_n(X^n)\psi(X_{n+1})\right| Y^{n+1}\right]\right] \\
	&= \E\left[\phi(Z) \psi_n(X^n)\psi(X_{n+1})\right].
\end{align*}
Again, the class of functions of the form $\Z \times \X^n \times \X \ni (z,x,x') \mapsto \phi(z)\psi_n(x)\psi(x')$, where $\phi$, $\psi_n$, and $\psi$ are as above, is convergence determining, and \eqref{pf:density3} follows.

By continuity of $\hat{g}^j$, the limit \eqref{pf:density1} implies that, for each $j$,
\begin{align}
\lim_{k\rightarrow\infty}&(Z,g^k_1(Y^1),\ldots,g^k_n(Y^n),\hat{g}^j(Y^{n+1},g^k_1(Y^1),\ldots,g^k_n(Y^n))) \nonumber \\
	&= (Z,X_1,\ldots,X_n,\hat{g}^j(Y^{n+1},X_1,\ldots,X_n)) \nonumber \\
	&= (Z,X^n,\hat{g}^j(Y^{n+1},X^n)). \label{pf:density4}
\end{align}
Combining the two limits \eqref{pf:density3} and \eqref{pf:density4}, we may find a subsequence $j_k$ such that 
\begin{align*}
\lim_{k\rightarrow\infty}&(Z,g^{j_k}_1(Y^1),\ldots,g^{j_k}_n(Y^n),\hat{g}^k(Y^{n+1},g^{j_k}_1(Y^1),\ldots,g^{j_k}_n(Y^n))) = (Z,X^n,X_{n+1}).
\end{align*}
Define $h^k_\ell := h^{j_k}_\ell$ for $\ell=1,\ldots,n$ and $h^k_{n+1}(Y^{n+1}) := \hat{g}^k(Y^{n+1},g^{j_k}_1(Y^1),\ldots,g^{j_k}_n(Y^n))$ to complete the induction.
\hfill\qedsymbol

\section{The lattice of stopping times}

In this section, we prove that the set $\SS$ of stopping times defined in Section \ref{se:proofs-supermodular} is a complete lattice. Recall that $\SS$ is defined as the set of (equivalence classes of a.s. equal) random times $\tau$ defined on the probability space $\Omega^{\mathrm{com}} \times \Omega^{\mathrm{ind}}$, which are stopping times with respect to the filtration $\FF^{\text{sig}}$. Recall that the \emph{essential supremum} of a family $\Phi$ of random variables is defined as the minimal (with respect to a.s. order) random variable exceeding a.s. each element of $\T$:

\begin{theorem}[Theorem A.33 of \cite{follmer-schied-book}] \label{th:essentialsup}
Let $\Phi$ be a set of real-valued random variables. Then there exists a unique (up to a.s. equality) random variable $Z = \esssup \Phi$ such that $Z \ge X$ a.s. for each $X \in \Phi$, and also $Z \le Y$ a.s. for every random variable $Y$ satisfying $Y \ge X$ a.s. for every $X \in \Phi$. Moreover, there exists a countable set $\Phi_0 \subset \Phi$ such that $Z = \sup_{X \in \Phi_0}X$ a.s.
\end{theorem}
\begin{proof}
Existence and uniqueness is stated in \cite[Theorem A.33]{follmer-schied-book}, and the proof therein constructs the desired $\Phi_0$.
\end{proof}

The essential infimum is defined analogously, or simply by $\essinf\Phi = -\esssup(-\Phi)$.

\begin{theorem} \label{th:stoppingtime-completelattice}
The set $\SS$ is a complete lattice.
\end{theorem}
\begin{proof}
Fix a set $\Phi \subset \SS$. Define $Z = \esssup \Phi$ and find a countable set $\{\tau_n : n \ge 1\} \subset \Phi$ such that $Z = \sup_n\tau_n$ a.s. Define $\sigma_n = \max_{k=1,\ldots,n}\tau_k$, so that $\sigma_n$ is an increasing sequence of stopping times with $\sigma_n \uparrow Z$ a.s. The increasing limit of a sequence of stopping times is again a stopping time \cite[Theorem IV.55(b)]{dellacheriemeyerA}, so $Z \in \SS$.

A similar argument applies to show that the essential infimum of $\Phi$ is also a stopping time, and the only difference is that this step crucially uses the right-continuity of the filtration $\overline{\FF}^{\text{sig}}$; indeed, while the supremum of a sequence of stopping times is always a stopping time, the infimum of a sequence of stopping times is only guaranteed to be a stopping time if the underlying filtration is right-continuous \cite[Theorem IV.55(c)]{dellacheriemeyerA}.
\end{proof}

\bibliographystyle{amsplain}
\bibliography{MFSG-bib}

\end{document}